\documentclass{amsart}
\usepackage{amssymb,amsmath,xcolor, stmaryrd,tikz,caption,subcaption, float}
\usepackage[colorlinks=true, linkcolor=blue, citecolor=green,]{hyperref}
\usetikzlibrary{shapes,arrows}

\allowdisplaybreaks
\newcommand{\comment}[1]{}

\renewcommand{\tilde}{\widetilde}

\newcommand{\Z}{\mathbb{Z}}

\newcommand{\R}{\mathbb{R}}

\DeclareMathOperator{\Inv}{\mathrm{Inv}}

\newcommand{\sm}{\mathrm{small}}
\newcommand{\bi}{\mathrm{big}}

\newcommand{\aff}{\mathrm{aff}}
\newcommand{\fin}{\mathrm{fin}}
\newcommand{\h}{\mathrm{ht}}
\newcommand{\dom}{\mathrm{dom}}
\newcommand{\tv}{\tilde{v}}
\newcommand{\tw}{\tilde{w}}
\newcommand{\tal}{\tilde{\alpha}}
\newcommand{\tbe}{\tilde{\beta}}

\newtheorem{lem}{Lemma}
\newtheorem{thm}[lem]{Theorem}
\newtheorem{prop}[lem]{Proposition}
\newtheorem{cor}[lem]{Corollary}

\theoremstyle{remark}
\newtheorem{rem}[lem]{Remark}
\newtheorem{ex}[lem]{Example}

\numberwithin{equation}{section}
\numberwithin{lem}{section}

\theoremstyle{definition}
\newtheorem{definition}{Definition}[section]

\begin{document}

\title[Double Affine Bruhat Order]{Double Affine Bruhat Order}
\author{A. Welch}
%\date{\today}
\maketitle

\begin{center}

\end{center}

We classify cocovers and covers of a given element of the double affine Weyl semigroup $W$ with respect to the Bruhat order, specifically when $W$ is associated to a finite root system that is irreducible and simply laced. We show two approaches: one extending the work of Lam and Shimozono \cite{LS}, and its strengthening by Mili\'cevi\'c \cite{Mi}, where cocovers are characterized in the affine case using the quantum Bruhat graph of $W_{\fin}$, and another, which takes a more geometrical approach by using the length difference set defined by Muthiah and Orr \cite{MO}.

\section{Introduction}

The affine Weyl group $W_{\aff}$ is a Coxeter group that can be created from the finite Weyl group $W_{\fin}$. It has many of the same properties as the finite Weyl group, and both groups have been extensively studied. This paper looks at the double affine Weyl semigroup, $W$, which is created from the affine Weyl group. In particular, we examine the Bruhat order on $W$ and classify covers and cocovers when the associated finite root system $\Phi_{\fin}$ is irreducible and simply laced.

Braverman, Kazhdan, and Patnaik introduced a pre-order on $W$ in \cite{BKP} while examining Iwahori-Hecke algebras for affine Kac-Moody groups. Given $x \in W$ and $\alpha$ a positive double affine root, they defined a preorder with generating relations: $x \geq s_{\alpha}x$ if and only if $x^{-1}(\alpha) < 0$. 

They called this preorder the Bruhat preorder and conjectured that it was an order (it was known that in the finite and affine case it is an order). In \cite{M} it was shown that the preorder is in fact an order and in \cite{MO} it was shown that the order coincides with the order generated by the relations: $x \geq xs_{\alpha}$ if and only if $\ell(x) \geq \ell( x s_{\alpha}).$

Further, Muthiah and Orr \cite{MO} related the cocover and cover relationships to a difference in lengths when the finite root system in question is irreducible and simply laced. The showed that for $\alpha$ a positive double affine root, $s_{\alpha}x$ is a cocover of $x$ if and only if $\ell(x) = \ell(s_{\alpha}x) + 1.$

This paper focuses on classifying cocovers and covers of a given element of $W$ with respect to the Bruhat order and using our classifications to better understand the Bruhat intervals.

First, we classify cocovers by extending the work done by Lam and Shimozono \cite{LS} and further strengthened by Mili\'cevi\'c \cite{Mi}, where cocovers were classified in the affine case. 

\begin{thm}(Theorem \ref{thm: 1} below)
Let $x= X^{\tv\zeta} \tw \in W$ where $\zeta$ is dominant and $\tv, \tw \in W_{\aff}$. Let $y = s_\alpha x$ where $\alpha = -\tv \tal + j \pi$ is a positive double affine root. Choose $M$ so that $\ell(\tw), \ \ell(s_{\tv \tal} \tw) \leq M$, and assume that $ \langle \zeta, \alpha_i \rangle \geq 2(M+1)$ for $i = 0, 1, \dotsc, n$. Then $y$ is a cocover of $x$ if and only if one of the following holds:

\begin{enumerate}

\item  $\ell(\tv) = \ell(\tv s_{\tal}) + 1$ and $j = 0$ so $y = X^{\tv s_{\tal}\zeta} s_{\tv \tal} \tw.$
 
\item $\ell(\tv) = \ell(\tv s_{\tal}) + 1 - \langle \tal , 2\rho \rangle$ and $j =1$ so $y = X^{\tv s_{\tal}(\zeta - \tal)} s_{\tv \tal} \tw.$

\item  $\ell(\tw^{-1}\tv s_{\tal})  =  \ell(\tw^{-1}\tv) + 1$ and $j = \langle \zeta, \tal \rangle$ so $y = X^{\tv\zeta} s_{\tv \tal} \tw.$

\item  $\ell(\tw^{-1}\tv s_{\tal}) = \ell(\tw^{-1}\tv) + 1 - \langle \tal, 2\rho \rangle$ and $j = \langle \zeta, \tal \rangle - 1$ so $y = X^{\tv (\zeta - \tal)} s_{\tv \tal} \tw.$

\end{enumerate}
\end{thm}

This approach allows us to classify cocovers by using the quantum Bruhat graph of $W_{\aff}$, but in doing so, we must impose length bounds on the affine parts of the elements under consideration. 

Second, we use the length difference set defined by Muthiah and Orr \cite{MO} to better our classification by ridding ourselves of the length bounds. 

\begin{thm}(Theorem \ref{thm: 2nd} below)
Let $x = X^{\tv \zeta} \tw$ and $y = s_{\alpha}x$ where $\alpha = -\tv \tal + j \pi$ is a positive double affine root and $\langle \zeta, \alpha_i \rangle > 2$ for $i = 0, 1, \dotsc, n$. Then $y$ is a cocover of $x$ if and only if one of the following holds: 

\begin{enumerate}
\item $j = 0$ and $\ell(\tv) = \ell(\tv s_{\tal}) + 1$.
\item $j  =1$ and $\ell(\tv) = \ell(\tv s_{\tal}) + 1 - \langle \tal, 2\rho \rangle$.
\item $j = \langle \zeta, \tal \rangle$ and $\ell(\tw^{-1} \tv s_{\tal}) = \ell(\tw^{-1} \tv) + 1.$
\item $j = \langle \zeta, \tal \rangle - 1$ and $\ell(\tw^{-1} \tv s_{\tal}) = \ell(\tw^{-1} \tv) + 1 - \langle \tal, 2\rho \rangle$.
\end{enumerate}
\end{thm}

This approach allows us to show that there are finitely many cocovers and covers for a given element $x \in W$. Additionally, it allows us to prove the following corollary concerning Bruhat intervals.

\begin{cor}(Corollary \ref{fin} below)
Let $x, y \in W$ such that $y \leq x$. Then the double affine Bruhat interval $[y, x]$ will be finite. 
\end{cor}

%%%%%%%%%%%%%%%%%%%%%%%%%
\section{Background}

Let $W_{\fin}$ denote our finite Weyl group with associated root lattice $\Phi_{\fin}$ irreducible and simply laced. Let $W_{\aff}$ denote the affine Weyl group created from the semidirect product of the translation group associated to $Q = \Z\Phi_{\fin}$ with $W_{\fin}$. Because $\Phi_{\fin}$ is simply laced, we have a pairing $\langle \  , \ \rangle$ such that $\langle \alpha, \alpha \rangle = 2$ for all $\alpha \in \Phi_{\fin}$. This allows us to identify $\Phi_{\fin}$ with $\Phi_{\fin}^{\vee}$, the set of coroots.

Let $P_{\fin}$ be the finite weight lattice and $X = P_{\fin} \oplus \mathbb{Z}\delta \oplus \mathbb{Z} \Lambda_0$, the affine weight lattice. Given $\zeta = \mu + m\delta + l\Lambda_0 \in X$, we call $l$ the \textbf{level} of $\zeta$ and denote it by lev$(\zeta) = l$. 

\begin{prop}\label{prop: 4}\cite[(6.5.2)]{K}
Let $\lambda \in Q$ and $w \in W_{\fin}$. The action of $W_{\aff}$ on $X$ is defined by 
$$Y^{\lambda}w(\mu + m\delta + l \Lambda_0) = w(\mu) + l\lambda + (m - \langle w(\mu), \lambda \rangle - l \frac{\langle \lambda, \lambda \rangle}{2}) \delta + l \Lambda_0.$$
\end{prop}

%\begin{definition}
Let $X_{\dom}$ be the set of all dominant elements of $X$. Then the \textbf{Tits cone} $\mathcal{T}$ is given by $$\mathcal{T} = \cup_{w \in W_{\aff}} w(X_{\dom}).$$
%\end{definition}

In this definition, we see that the Tits cone is the subset of $X$ containing all elements that can be made dominant by some element of $W_{\aff}$.

Alternatively, we can view the Tits cone as a union of two sets, with one set containing elements of level zero and the other containing the elements with positive level.

\begin{prop}\cite[Prop 5.8(b)]{K}
$$\mathcal{T} = \{m\delta : m \in \mathbb{Z}\} \cup \{ \mu + m\delta + l \Lambda_0 : \mu \in P,  \ m \in \mathbb{Z}, \  l \in \mathbb{Z}_{>0}\}.$$
\end{prop}

Note that $\mathcal{T}$ contains all the imaginary roots (roots of the form $m\delta$) and all the roots with $l > 0$, but it contains no elements with a negative level.

%\begin{definition}
We define the \textbf{double affine Weyl semigroup} $W$ to be the semidirect product of the the translation semigroup associated to $\mathcal{T}$ with $W_{\aff}$:
\begin{align*}
W &= \mathcal{T} \rtimes W_{\aff}\\
&= \{X^{\zeta} \tw : \zeta \in \mathcal{T}, \tw \in W_{\aff} \}\\
&=\{X^{\zeta} Y^{\lambda} w : \zeta \in \mathcal{T}, \lambda \in Q, w \in W_{\mathrm{fin}} \}.
\end{align*}
%\end{definition}

\begin{rem}
This is a semigroup, but not a group, as it is not closed under inverses.
\end{rem}

For simplicity, we will use $\textrm{lev}(x)$  to denote the level of the $X$-weight of $x \in W$ (i.e. if $x = X^\zeta Y^\lambda w \in W$ then $\textrm{lev}(x) = \textrm{lev}(\zeta)$).

%%%%%%%%%%%%%%%%%%%%%%%%%%%%%%%%%%%%%%%%%%%%%%%

\subsection{Roots and Reflections}

%\begin{definition}
Define $Q_{\textrm{daff}} = \mathbb{Z}\Phi_{\fin} \oplus \mathbb{Z}\delta \oplus \Z \pi$. The set of \textbf{double affine roots} is given by
$$\Phi = \{ \tal + j \pi \in Q_{\textrm{daff}}: \tal \in \Phi_{\aff}, j \in \mathbb{Z}\} = \{ \nu + r\delta + j \pi : \nu \in \Phi_{\fin}, r, j \in \mathbb{Z} \}.$$

Let $\tal = \nu + r\delta$ be an affine root. We say that a double affine root $\alpha = \tal + j \pi$ is \textbf{positive} if $\tal > 0$ and $j \geq 0$ or $\tal < 0$ and $j > 0$. Similarly, we say that a double affine root $\alpha = \tal + j \pi$ is \textbf{negative} if $\tal < 0$ and $j \leq 0$ or $\tal > 0$ and $j < 0$. For our purposes, we will consider $\pi$ to be a placeholder like $\delta$ for the affine root. 
%\end{definition}

Each double affine root $\alpha = \nu + r\delta + j\pi$, can be associated to a reflection $s_{\alpha}$. Let $\tal = \nu + r\delta$. Then define:
\begin{align*}
s_{\alpha}  &= s_{\tal + j\pi}\\
&= X^{-j \tal} s_{\nu + r\delta}\\
&= X^{-j \tal} Y^{-r \nu} s_{\nu}.
\end{align*}

%Note that $s_{\alpha}$ is an element of $Q_{\aff} \rtimes W_{\aff}.$

\begin{rem}
If $\alpha = \nu + r\delta + j\pi$ is a double affine root, and $j \neq 0$, then  $s_{\alpha}$ is not an element of $W$. Instead, $s_{\alpha}$ is an element of $X \rtimes W_{\aff}$, which contains $W$ as a sub-semi-group.

Consider
$$s_{\nu + r\delta + j\pi} = X^{-j(\nu + r \delta)} Y^{-r\nu} s_{\nu}$$

\noindent with $j \neq 0$. Then $s_{\nu + r\delta + j\pi}$ is not an element of $W$ because $-j(\nu + r d)$ is not in $\mathcal{T}$; however, when we consider $x = X^{\zeta}\tw \in W $with $\textrm{lev}(x) > 0,$ $xs_{\nu + r\delta + j\pi}$ is an element of the double affine Weyl semigroup. 

\end{rem}

\begin{rem}
The semigroup $W$ is not generated by reflections.

Consider $x = X^{\mu + m\delta + l\Lambda_0} \in W$ with lev$(x) > 0$. Then $x$ cannot be written as a product of reflections because the reflections contain no $X^{l \Lambda_0}$ part.
\end{rem}

\begin{prop}\label{action} Let $\zeta \in X$ and $\tw \in W_{\aff}$. $X \rtimes W_{\aff}$ acts on $\Phi$ by 
$$X^{\zeta} \tw (\tal + j \pi) = \tw(\tal) + (j - \langle \zeta, \tw(\tal) \rangle) \pi.$$

This is similar to the action defined for $W_{\aff}$ on $\Phi_{\aff}$. Letting $\zeta = \mu + m \delta + l\Lambda_0$ and $\tw = Y^{\lambda} w,$ we can expand this to
\begin{align*}
X^{\zeta} Y^{\lambda} w & (\alpha +  r\delta + j \pi) =  Y^{\lambda} w (\alpha + r\delta) + (j - \langle \zeta, Y^{\lambda} w (\alpha + r\delta) \rangle ) \pi\\
= & \ Y^{\lambda} w (\alpha + r\delta) + (j - \langle\mu + m \delta + l \Lambda_0, Y^{\lambda} w (\alpha + r\delta) \rangle ) \pi \\
= & \ w(\alpha) + (r - \langle \lambda, w(\alpha) \rangle ) \delta + (j - \langle \mu + m \delta + l \Lambda_0, w(\alpha) + (r - \langle \lambda, w( \alpha) \rangle ) \delta \rangle ) \pi \\
= &\  w(\alpha) + (r - \langle \lambda, w(\alpha) \rangle ) \delta + (j - \langle \mu, w(\alpha) \rangle - l(r - \langle \lambda, w( \alpha) \rangle )) \pi \\
= & \ Y^{\lambda} w (\alpha + r\delta) + (j - \langle \mu, w(\alpha) \rangle - l(r - \langle \lambda, w( \alpha) \rangle )) \pi.
\end{align*}
\end{prop}

\begin{prop}
Let $\alpha$ and $\beta$ be double affine roots. Then $s_{\alpha}(\beta)$ as defined in Proposition \ref{action} is the same as 
$$s_{\alpha}(\beta) = \beta - \langle \alpha, \beta \rangle \alpha.$$
\end{prop}

\begin{proof}
When we expand the action defined in Proposition \ref{action}, we can see that the two actions are the same:
\begin{align*}
s_{\alpha}(\beta) & = X^{-j(\nu + r\delta)}Y^{-r\nu}s_{\nu}(\gamma + p\delta + q \pi)\\
& = s_{\nu}(\gamma) + (p + \langle r\nu, s_{\nu}(\gamma) \rangle)\delta + (q + \langle j(\nu + r\delta), s_{\nu}(\gamma) + (p + \langle r\nu, s_{\nu}(\gamma) \rangle)\delta \rangle)\pi\\
& = s_{\nu}(\gamma) + (p - r\langle \nu, \gamma \rangle)\delta + (q + j\langle \nu + r\delta, s_\nu(\gamma) \rangle)\pi \\
& = (\gamma - \langle \nu, \gamma \rangle \nu) + (p- r\langle \nu, \gamma \rangle)\delta + (q - j \langle \nu, \gamma \rangle)\pi\\
& = \gamma + p\delta + q\pi - \langle \nu, \gamma \rangle (\nu + r\delta + j\pi)\\
& = \beta- \langle \alpha, \beta \rangle \alpha
\end{align*}
\end{proof}

%%%%%%%%%%%%%%%%%%%%%%%%%%%%%%%%%%%%%%%%%%%%%%%

\subsection{Length Function}

When considering elements of $W$, it no longer makes sense to define a length function based on reduced words (as is done for $W_{\fin}$ or $W_{\aff}$) because not every element of $w$ can be expressed as a product of simple reflections. Instead, we use the length function defined in \cite{M}. 

Let $\rho$ be the sum of the affine fundamental weights (which we choose now and keep consistent throughout the paper).

%\begin{definition}
Let $x = X^{\zeta}\tw$ be an element of $W$. Then the \textbf{length} of $x$ is defined by \cite{M} to be
$$\ell(x) = \langle \zeta_{+}, 2\rho \rangle + |\{ \tal \in \Inv(\tw^{-1}) : \langle \zeta, \tal \rangle \leq 0 \}| - |\{ \tal \in \Inv(\tw^{-1}) : \langle \zeta, \tal \rangle > 0 \}|,$$

\noindent where $\zeta_{+}$ is the dominant element associated to $\zeta$ and $\tal = \nu + r\delta$ is an affine root. We break this into big and small parts by defining the \textbf{big length} as 
$$\ell_{\bi}(x) = \langle \zeta_{+}, 2\rho \rangle$$ 

\noindent and the \textbf{small length} as
$$\ell_{\sm}(x) =  |\{ \tal \in \Inv(\tw^{-1}) : \langle \zeta, \tal \rangle \leq 0 \}| - |\{ \tal \in \Inv(\tw^{-1}) : \langle \zeta, \tal \rangle > 0 \}|.$$
%\end{definition}

From our definition of $\ell$, we can see why we must use $\mathcal{T}$ and not all of $X$ when defining $W$. Recall that $\mathcal{T}$ contains all elements of $X$ that can be made dominant. We need the $X$-weight of $x \in W$ to be made dominant when calculating the length since we use $\ell_{\bi}(X^{\zeta}\tw)  = \langle \zeta_{+} , 2\rho \rangle,$ where $\zeta_{+}$ is the dominant element of $X$ associated to $\zeta$. 

\begin{rem} 
For $\tw = Y^{\lambda} w \in W$, $\ell(\tw) = \ell(X^{0} \tw) = \ell_{\aff} (\tw),$ where $\ell_{\aff}$ is the Coxeter length function on $W_{\aff}$. 
\end{rem}

Before ending our discussion of the length function, we need a proposition that splits the length of an element $x \in W$ into the sum of two lengths, the first considering only the translation part of $x$ and the second considering only the affine part of $x$. This way of re-writing the length function will be fundamental when proving our classification theorems. 

\begin{prop}\label{prop: 6}
Let $\zeta \in \mathcal{T}$ be regular and dominant and let $x = X^{\tv \zeta}\tw$ where $\tw, \tv \in W_{\aff}$. Then 
\begin{align*}
\ell(x) & = \ell(X^\zeta) - \ell(\tv^{-1}\tw) + \ell(\tv)\\
&= \langle \zeta, 2\rho \rangle - \ell(\tw^{-1}\tv) + \ell(\tv).
\end{align*}
\end{prop}

Before we can prove Proposition \ref{prop: 6}, we need the following lemmas.

\begin{lem}\label{aff_len}
Let $x, y \in W_\aff$. Then 
\begin{align*}
\ell(xy) &= \ell(x) + \ell(y) - 2|\{\Inv(y) \cap -y^{-1} \Inv(x) \}|\\
& = \ell(x) + \ell(y) - 2|\{ \alpha \in \Inv(y) : \alpha \notin \Inv(xy) \}|.
\end{align*}
\end{lem}

\begin{proof}
Let $\alpha \in \Inv(xy)$. There are two possibilities:

\begin{enumerate}
\item  $\alpha > 0, \ y(\alpha) < 0,$ and $xy(\alpha) < 0$ 
\item $\alpha > 0, \ y(\alpha) > 0,$ and $xy(\alpha) < 0$. 
\end{enumerate}

So $\Inv(xy) \subset y^{-1}\Inv(x) \ \sqcup \  \Inv(y)$ (this is a disjoint union because if $\alpha \in y^{-1}\Inv(x)$, then $y(\alpha) > 0$ and so $\alpha \notin \Inv(y)$). In general, this is a proper subset because there could be $\alpha \in \Inv(y)$ such that $-y(\alpha) \in \Inv(x)$ (so $\alpha \notin \Inv(xy)$), or there could be $\alpha < 0$ such that $y(\alpha) \in \Inv(x)$ (so $\alpha \in y^{-1}\Inv(x)$ but $\alpha \notin \Inv(xy)$).

So $|\Inv(xy)| \leq |y^{-1}\Inv(x)| + |\Inv(y)|$, and to find an exact representation of $|\Inv(xy)|$, we must subtract $|\{\alpha \in \Inv(y) : -y(\alpha) \in \Inv(x) \}| = |\Inv(y) \cap -y^{-1}\Inv(x)|$ and $|\{\alpha < 0 : y(\alpha) \in \Inv(x) \} | = |\{\beta > 0 : -y(\beta) \in \Inv(x)\}| = |\Inv(y) \cap -y^{-1}\Inv(x)|$.

Using $|\Inv(x)| = |y^{-1} \Inv(x)|$, we have $|\Inv(xy)| = |\Inv(x)| + |\Inv(y)| - 2|\Inv(y) \cap - y^{-1} \Inv(x)| =  |\Inv(x)| + |\Inv(y)| - 2|\{ \alpha \in \Inv(y) : \alpha \notin \Inv(xy) \}|$.
\end{proof}

\begin{lem}\label{aff_len_sec}
Let $x, y$ be elements of $W_{\aff}$. Then 
$$\ell(xy) = \ell(x) + \ell(y) - 2|\Inv (x) \cap \Inv(y^{-1})|.$$
\end{lem}

\begin{proof}
Using Lemma $\ref{aff_len}$, this is equivalent to showing  $|\{ \gamma\in \Inv(x) \cap \Inv(y^{-1})\}| = |\{ \gamma \in \Inv(y) : -y(\gamma) \in \Inv(x) \}|$.

We create a bijection by mapping $\gamma \in \{ \gamma \in \Inv(y) : -y(\gamma) \in \Inv(x) \}$ to $-y(\gamma) \in \{ \gamma\in \Inv(x) \cap \Inv(y^{-1})\}$, so the sets have the same size.
\end{proof}

Now we will prove Proposition \ref{prop: 6}.

\begin{proof}[Proof of Proposition \ref{prop: 6}]

We have 
\begin{align*}
\ell(X^{\tv \zeta}& \tw) = \ell_{\bi}(X^{\tv \zeta} \tw) + \ell_{\sm}(X^{\tv \zeta} \tw)\\
&=\langle \zeta, 2\rho \rangle + |\{ \gamma \in \Inv(\tw^{-1}) : \langle \tv \zeta, \gamma \rangle \leq 0 \}| - |\{ \gamma \in \Inv(\tw^{-1}) : \langle \tv \zeta, \gamma \rangle > 0 \}|.
\end{align*}

We need to show that $\ell(\tv) - \ell(\tw^{-1}\tv) =   \ell_{\sm}(X^{\tv \zeta} \tw)$.

Note that $\langle \tv \zeta, \gamma \rangle = \langle \zeta, \tv^{-1}(\gamma) \rangle$ and since $\zeta$ is dominant and regular, $\langle \zeta, \tv^{-1}(\gamma) \rangle > 0$ if and only if $\tv^{-1}(\gamma) > 0$. Similarly, $\langle \zeta, \tv^{-1}(\gamma) \rangle < 0$ if and only if $\tv^{-1}(\gamma) < 0$ (since $\zeta$ is regular, we know $\langle \zeta, \tv^{-1}(\gamma) \rangle \neq 0$).

So  $\{ \gamma \in \Inv(\tw^{-1}) : \langle \tv \zeta, \gamma \rangle \leq 0 \} = \Inv(\tw^{-1}) \cap \Inv(\tv^{-1})$ and $\{ \gamma \in \Inv(\tw^{-1}) : \langle \tv \zeta, \gamma \rangle > 0 \} = \{ \gamma \in \Inv(\tw^{-1}) :  \tv^{-1}(\gamma) > 0\}$.

By Lemma \ref{aff_len_sec}, $\ell(\tw^{-1}\tv) = \ell(\tw^{-1}) + \ell(\tv) - 2| \Inv(\tw^{-1}) \cap \Inv(\tv^{-1})|$ so $2|\Inv(\tw^{-1}) \cap \Inv(\tv^{-1})| -\ell(\tw^{-1}) = \ell(\tv) - \ell(\tw^{-1}\tv)  $. Therefore,
\begin{align*}
\ell_{\sm}(X^{\tv \zeta} \tw)& = |\{ \gamma \in \Inv(\tw^{-1}) : \langle \tv \zeta, \gamma \rangle \leq 0 \}| - |\{ \gamma \in \Inv(\tw^{-1}) : \langle \tv \zeta, \gamma \rangle > 0 \}| \\
&= |\Inv(\tw^{-1}) \cap \Inv(\tv^{-1})| - |\{ \gamma \in \Inv(\tw^{-1}) :  \tv^{-1}(\gamma) > 0\}| \\
& =|\Inv(\tw^{-1}) \cap \Inv(\tv^{-1})| - ( |\Inv(\tw^{-1})| - |\Inv(\tw^{-1}) \cap \Inv(\tv^{-1})|)\\
&=2|\Inv(\tw^{-1}) \cap \Inv(\tv^{-1})| - \ell(\tw^{-1})\\
&=\ell(\tv) - \ell(\tw^{-1}\tv).
\end{align*} \end{proof}

\subsection{Bruhat Order}

Given $x \in W$ with lev$(x) > 0$ and $\alpha$ a positive double affine root, \cite[5, Section B.2]{BKP} defined $x \rightarrow xs_{\alpha}$ if $x(\alpha) > 0$ (we exclude $x  \in W$ with $\textrm{lev}(x) = 0$ because in that case, $xs_{\alpha}$ is not always in $W$). They defined the \textbf{double affine Bruhat preorder} to be the preorder generated by these relations, (that is, $x \leq y$ if there is some chain $x \rightarrow xs_{\alpha_1} \rightarrow \dotsm \rightarrow y$), and they conjectured that it was an order. In \cite{M} it was shown that the preorder is in fact an order, and in \cite{MO} it was shown that this order coincides with the order generated by the relations: $x \rightarrow xs_{\alpha}$ if $\ell(x) \leq \ell( x s_{\alpha}).$ When multiplying on the left, we use the relation $x \rightarrow s_{\alpha} x$ if $\ x^{-1}(\alpha) > 0$.

Let $x, y \in W$. Then $y$ is said to be a cover of $x$ if $x < y$ and there is no $z \in W$ such that $x < z < y$. Similarly, $y$ is said to be a cocover of $x$ if $y < x$ and there is no $z \in W$ such that $y < z < x$.

We are interested in classifying covers and cocovers for a fixed $x \in W$ where the associated finite root system $\Phi_{\fin}$ is irreducible and simply laced. Muthiah and Orr \cite{MO} proved the following theorem that will allow us to identify cocovers and covers by a difference in length.

\begin{thm}
 \cite[Thm 1.6]{MO}
For $\alpha$ a positive double affine root and $x \in W$ with $\textrm{lev}(x) > 0$, $x s_{\alpha} \ \text{is a cover of} \ x$ if and only if $\ell(x) = \ell(x s_{\alpha}) - 1.$ 
\end{thm}

\noindent We can similarly say that $xs_{\alpha}$ is a cocover of $x$ if and only if $\ell(x) = \ell(xs_{\alpha}) + 1$.

%%%%%%%%%%%%%%%%%%%%%%%%%%%%%%%%%%%%%%%%%%%%%%%

\section{First Method for Classifying Cocovers}

We wish to determine the cocovers of a given element $x$ in $W$. To do this we extend \cite[Prop 4.1]{LS} of Lam and Shimozono and the further strengthening \cite[Prop 4.2]{Mi} by Mili\'cevi\'c  that classifies cocovers of an element $x$ in $W_{\aff}$ by using the quantum Bruhat graph of $W_{\fin}$. 
For the remainder of the paper, when considering elements $x \in W$, we will assume lev$(x) > 0$. Additionally, recall that we are only considering $W$ where the associated $\Phi_{\fin}$ is irreducible and simply laced.

\subsection{Classification}

\begin{thm}\label{thm: 1}
Let $x= X^{\tv\zeta} \tw \in W$ where $\zeta$ is dominant and $\tv, \tw \in W_{\aff}$. Let $y = s_\alpha x$ where $\alpha = -\tv \tal + j \pi$ is a positive double affine root. Choose $M$ so that $\ell(\tw), \ \ell(s_{\tv \tal} \tw) \leq M$, and assume that $ \langle \zeta, \alpha_i \rangle \geq 2(M+1)$ for $i = 0, 1, \dotsc, n$. Then $y$ is a cocover of $x$ if and only if one of the following holds:

\begin{enumerate}

\item  $\ell(\tv) = \ell(\tv s_{\tal}) + 1$ and $j = 0$ so $y = X^{\tv s_{\tal}\zeta} s_{\tv \tal} \tw.$
 
\item $\ell(\tv) = \ell(\tv s_{\tal}) + 1 - \langle \tal , 2\rho \rangle$ and $j =1$ so $y = X^{\tv s_{\tal}(\zeta - \tal)} s_{\tv \tal} \tw.$

\item  $\ell(\tw^{-1}\tv s_{\tal})  =  \ell(\tw^{-1}\tv) + 1$ and $j = \langle \zeta, \tal \rangle$ so $y = X^{\tv\zeta} s_{\tv \tal} \tw.$

\item  $\ell(\tw^{-1}\tv s_{\tal}) = \ell(\tw^{-1}\tv) + 1 - \langle \tal, 2\rho \rangle$ and $j = \langle \zeta, \tal \rangle - 1$ so $y = X^{\tv (\zeta - \tal)} s_{\tv \tal} \tw.$

\end{enumerate}
\end{thm}

Before we can prove this theorem, we will need the following lemmas, which are inspired by the proofs  of \cite[Prop 4.1]{LS} and \cite[Prop 4.2]{Mi}. 

\begin{lem}\label{lem: 2}
Define $f(j) = \ell(X^{\tv(\zeta - j \tbe)})$ where $j \in \Z$, $\zeta \in \mathcal{T}$, $\tv \in W_{\aff}$ and $\tbe$ is an affine root. Then $f(j)$ is a convex function.
\end{lem}

\begin{proof}
Define $S(t) = \{ \gamma \in \Phi^{+} : \langle \zeta , \tv^{-1}(\gamma)  \rangle - t \langle \tbe, \tv^{-1}(\gamma)  \rangle < 0 \}$ where $t \in \R$. Let $\tw \in W_{\aff}$ such that $\tw(\tv(\zeta -j\tbe))$ is dominant. Then $S(t)$ is a finite set contained in $\Inv(\tw)$. Using \cite[Prop 3.10]{MO}, we have 
$$f(j)= \langle \zeta - j\tbe, \tv^{-1}(2\rho)\rangle - \sum_{\gamma \in S_{(j)}} \langle \zeta - j\tbe, \tv^{-1}(2\gamma) \rangle.$$

Define $g(t)= - \sum_{\gamma \in S(t)} \langle \zeta , \tv^{-1}(2\gamma) \rangle - t \langle \tbe, \tv^{-1}(2\gamma) \rangle$ for $t \in \mathbb{R}$. We show $g(t)$ is convex by showing $g(tj_1 + (1-t)j_2) \leq tg(j_1) + (1-t)g(j_2)$ for $t \in [0, 1], j_i \in \Z$. This is trivially true for $t = 0, 1$. We consider $t \in (0,1)$. Note that $t$ and $(1 - t)$ are positive for these cases. Let
\begin{align*}
T = & \ S(tj_1 + (1 - t)j_2) \\
= & \ \{ \gamma \in \Phi^{+} : \ t \langle \zeta - j_1 \tbe, \tv^{-1}(\gamma) \rangle + (1 - t) \langle \zeta - j_2 \tbe, \tv^{-1}(\gamma) \rangle < 0 \}\\
A_1 = & \  \{ \gamma \in \Phi^{+} : \   \langle \zeta - j_1 \tbe, \tv^{-1}(\gamma) \rangle <  0,  \  \langle \zeta - j_2 \tbe, \tv^{-1}(\gamma) \rangle \geq 0 \} \subseteq S(j_1)\\
A_2 =  & \ \{ \gamma \in \Phi^{+} :  \ \langle \zeta - j_1 \tbe, \tv^{-1}(\gamma) \rangle \geq 0, \   \langle \zeta - j_2 \tbe, \tv^{-1}(\gamma) \rangle < 0 \} \subseteq S(j_2)\\
B = & \  \{ \gamma \in \Phi^{+} : \  \langle \zeta - j_1 \tbe, \tv^{-1}(\gamma) \rangle < 0,  \  \langle \zeta - j_2 \tbe, \tv^{-1}(\gamma) \rangle < 0 \} \subseteq S(j_1).
\end{align*}

All of these sets are finite because they can be contained in $S(j)$ for some $j$. Note $B = B \cap T$ and $T = (A_1 \cap T) \sqcup (A_2 \cap T) \sqcup B$. Also note $-\sum_{\gamma \in A_i \cap T} \langle \zeta - j_i \tbe, \tv^{-1}(2\gamma) \rangle \leq -\sum_{\gamma \in A_i} \langle \zeta - j_i \tbe, \tv^{-1}(2\gamma) \rangle$ for $i = 1, 2$. 

For $j$ an integer and $S$ some set, define $g(j, S) = \sum_{\gamma \in S} \langle \zeta - j \tbe, \tv^{-1}(2\gamma) \rangle$. Then
\begin{align*}
g(tj_1 + (1-t)j_2) & = -\sum_{\gamma \in T} \left( t \langle \zeta - j_1 \tbe, \tv^{-1}(2\gamma) \rangle + (1 - t) \langle \zeta - j_2 \tbe, \tv^{-1}(2\gamma) \rangle \right)\\
& =-tg(j_1, T) - (1- t)g(j_2, T)\\
& = -t(g(j_1, A_1 \cap T) + g(j_1, B) + g(j_1, A_2 \cap T)) \\
& \hspace{5mm} - (1 - t)(g(j_2, A_1 \cap T) + g(j_2, B) + g(j_2, A_2 \cap T)) \\
&\leq -t(g(j_1, A_1) + g(j_1, B)) - (1 - t)(g(j_2, A_2) + g(j_2, B))\\
&=-tg(j_1, A_1 \cup B) - (1 - t) g(j_2, A_2 \cup B)\\
&=-tg(j_1, S(j_1)) - (1 - t) g(j_2, S(j_2))\\
& =tg(j_1) + (1-t)g(j_2).
\end{align*}

So $f(j)=  \langle \zeta - j\tbe, \tv^{-1}(2\rho)\rangle + g(j)$ is convex as it is the sum of two convex functions ($\langle \zeta - j\tbe, \tv^{-1}(2\rho)\rangle$ is a function of the form $j \mapsto a + bj$ and so is convex).
\end{proof}

\begin{lem} \label{lem: 3}
Let $x= X^{\tv\zeta} \tw \in W$ where $\zeta$ is dominant. Let $y = s_\alpha x$ where $\alpha = -\tv \tal + j \pi$ is a positive double affine root. Choose $M$ so that $\ell(\tw), \ \ell(s_{\tv \tal} \tw) \leq M$. If $ \langle \zeta, \alpha_i \rangle \geq 2(M+1)$ for $i = 0, 1, \dotsc, n$ and if $y$ is a cocover of $x$, then $0 \leq j \leq M$ or $\langle \zeta, \tal \rangle - M \leq j \leq \langle \zeta, \tal \rangle$.
\end{lem}

\begin{proof}
By assumption, $y$ is a cocover of $x$ and $-\tv \tal + j \pi > 0$ so $x^{-1}(-\tv \tal + j \pi)  = -\tw \tv \tal + (j - \langle \zeta, \tal \rangle ) \pi < 0.$ This tells us $0 \leq  j \leq \langle \zeta, \tal \rangle$. Hence, $\tal $ is positive since $\zeta$ is dominant and regular. 

Consider $\alpha_i$ such that $\langle \tal, \alpha_i \rangle < 0$. Then $\langle \zeta - j \tal , \alpha_i \rangle \geq \langle \zeta, \alpha_i \rangle > 0$ since $\zeta$ is dominant and regular. Now consider the remaining $\alpha_i$. By assumption, $\langle \zeta, \alpha_i \rangle \geq 2(M + 1) \geq \langle \tal, \alpha_i \rangle (M + 1)$ (we are using the fact that $\langle \tal, \tbe \rangle \leq 2$ for all $\tal, \tbe \in \Phi_{\aff}$ as seen in \cite{B}), so $\langle \zeta - (M + 1) \tal, \alpha_i \rangle \geq 0$. If $j \leq M + 1$ then $\zeta - j \tal$ is dominant since $\langle \zeta - j \tal , \alpha_i \rangle \geq \langle \zeta - (M + 1) \tal, \alpha_i \rangle \geq 0$.

Let $j' = \langle \zeta, \tal \rangle - j$. If $j' \leq M + 1$, then $\zeta - j' \tal$ is dominant. So if $j \geq \langle \zeta, \tal \rangle - (M + 1)$, then $\zeta - (\langle \zeta, \tal \rangle - j) \tal$ is dominant.

Following Mili\'cevi\'c \cite[Proof of Prop 4.2]{Mi}, we re-write $y$ as 
\begin{align*}
y = s_{-\tv\tal + j \pi} x & = X^{j \tv\tal} s_{\tv\tal}X^{\tv\zeta}\tw\\
& = X^{j \tv\tal + s_{\tv\tal}\tv\zeta}s_{\tv\tal}\tw\\
& = X^{\tv(s_{\tal} \zeta + j \tal)} s_{\tv \tal} \tw\\
& = X^{\tv s_{\tal} ( \zeta - j \tal)} s_{\tv \tal} \tw\\
& = X^{\tv(\zeta - (\langle \zeta, \tal \rangle - j)\tal)} s_{\tv \tal} \tw.\\
\end{align*}
Define $f(j)= \ell(X^{\tv(\zeta - j \tal)}) = \ell(X^{\tv s_{\tal}(\zeta - j \tal)})$ as in Lemma \ref{lem: 2}. Then $f(j)$ is a convex function. Note also that $f(0) = \ell(X^{\tv \zeta}) = \langle \zeta, 2\rho \rangle = \ell( X^{\tv s_{\tal}(\zeta)}) = f(\langle \zeta, \tal \rangle)$, and for $j \in [0, \langle \zeta, \tal \rangle]$, $f(j)= f(\langle \zeta, \tal \rangle - j)$ since $\tv s_{\tal}(\zeta - j \tal) = \tv(\zeta - \langle \zeta, \tal \rangle \tal + j\tal) = \tv(\zeta - (\langle \zeta, \tal \rangle - j)\tal)$. 

\begin{center}
\includegraphics[scale = .4]{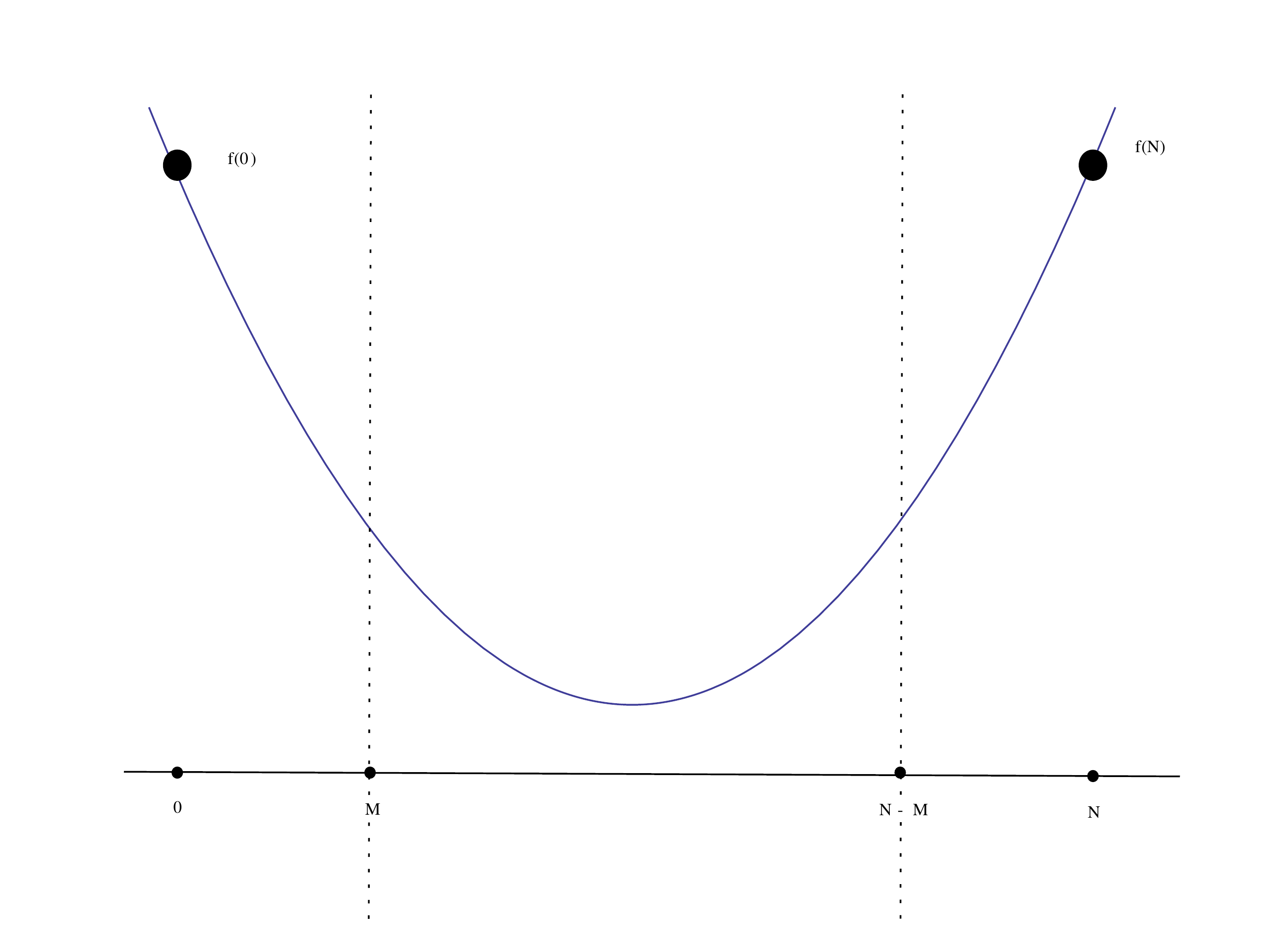}
\end{center}

The idea behind the proof, which comes from \cite[Proof of Prop 4.1]{LS}, is that $\ell(x) \approx f(0) = f(\langle \zeta, \tal \rangle)$ and $\ell(y) \approx f(j)$, so for $y$ to be a cocover of $x,$ either $f(j)$ is close to $f(0)$ (and so $j$ is close to 0) or $f(j)$ is close to $f(\langle \zeta, \tal \rangle)$ (and so $j$ is close to $\langle \zeta, \tal \rangle$). We illustrate with the picture above, using $N = \langle \zeta, \tal \rangle$ for simplicity.

We use the fact that $f(j)$ is convex and symmetric to approximate the shape of the graph. We will show that when $j$ moves beyond the dashed lines so that $M < j < N - M$, $f(j)$ becomes too far from $f(0)$ or $f(N)$ to allow $y$ to be a cocover of $x$.

Fix $j$ such that $M < j < N - M$. First we will make precise what we mean by $\ell(x) \approx f(0) = f(\langle \zeta, \tal \rangle)$ and $f(y) \approx f(j)$:
 
We have $|\ell(x) - f(0)| = |\ell(x) - f(\langle \zeta, \tal \rangle) | = |\ell(X^{\tv\zeta}\tw) - \ell(X^{\tv\zeta})| \leq \ell(\tw) \leq  M.$

For the given $j$, $|\ell(y) - f(j)| = |\ell(X^{\tv s_{\tal} ( \zeta - j \tal)} s_{\tv \tal} \tw) - \ell(X^{\tv s_{\tal}(\zeta - j \tal)})| \leq \ell(s_{\tv \tal}\tw) \leq M.$

Additionally, if $m \leq M + 1$, then $\zeta - m \tal$ is dominant, so $f(m) = \langle \zeta - m \tal, 2\rho \rangle = f(0) - m \langle \tal, 2\rho \rangle \leq f(0) - 2m$ (we are using the fact that $\langle \tal, 2\rho \rangle = 2\h(\tal) \geq 2$).

Lastly, note that $f(j) \leq f(M + 1)$ because of $f$'s symmetry and convexity. By putting these three things together, we have 
\begin{align*}
\ell(y) & \leq  f(j) + M\\
& \leq  f(M + 1) + M\\
& \leq f(0) - 2(M + 1)+ M\\
&=f(0)-M-2\\
&\leq  \ell(x) - 2.
\end{align*}

So if $M < j < N - M$,  $y$ cannot be covered by $x$. Therefore, $0 \leq j \leq M$ or $\langle \zeta, \tal \rangle - M \leq j \leq \langle \zeta, \tal \rangle$.

\end{proof}
%$y$ cannot be covered by $x$. Similarly, if $j = \langle \zeta, \tal \rangle - (M + 1),$ $y$ cannot be covered by $x$. And for $j \in (M + 1,  \langle \zeta, \tal \rangle - (M + 1))$, $f(j)\leq f(M + 1) = f( \langle \zeta, \tal \rangle - (M + 1)),$ so we again have that $y$ cannot be covered by $x$.

Now we can prove Theorem \ref{thm: 1}.

\begin{proof}[Proof of Theorem \ref{thm: 1}]
We follow the proof of \cite[Prop 4.2]{Mi}. Assume $y$ is covered by $x$. Recall from our proof of Lemma \ref{lem: 3} that
\begin{align*}
y &= s_{-\tv\tal + j \pi} x \\
& = X^{\tv s_{\tal} ( \zeta - j \tal)} s_{\tv \tal} \tw\\
& = X^{\tv(\zeta - (\langle \zeta, \tal \rangle - j)\tal)} s_{\tv \tal} \tw.\\
\end{align*}
By Lemma \ref{lem: 3}, we know that if $s_{\alpha}x$ is covered by $x$, then $0 \leq j \leq M$ or $\langle \zeta, \tal \rangle - M \leq j \leq \langle \zeta, \tal \rangle$. We will show that under these conditions either $\zeta - j \tal$ or $\zeta - (\langle \zeta, \tal \rangle - j) \tal$ must be dominant and regular.

By assumption, $\langle \zeta, \alpha_i \rangle \geq 2(M + 1) \geq \langle \tal, \alpha_i \rangle (M + 1)$, so $\langle \zeta - (M + 1) \tal, \alpha_i \rangle \geq 0$.  If $\langle \tal, \alpha_i \rangle < 0,$ then $\langle \zeta - j \tal , \alpha_i \rangle > \langle \zeta, \alpha_i \rangle > 0$. Otherwise, if $j \leq M$, $\langle \zeta - j \tal , \alpha_i \rangle > \langle \zeta - (M + 1) \tal, \alpha_i \rangle \geq 0$, so $\zeta - j \tal$ is dominant and regular. As in the proof of Lemma \ref{lem: 3}, we set $j' = \zeta - j \tal$ to show that $\zeta - (\langle \zeta, \tal \rangle - j) \tal$ is dominant and regular if $j \geq \langle \zeta, \tal \rangle - M$.

First, suppose $\zeta - j\tal$ is dominant and regular. Using Proposition \ref{prop: 6} and $\ell(y) = \ell(X^{\tv s_{\tal} ( \zeta - j \tal)} s_{\tv \tal} \tw):$
\begin{align*}
\ell(s_\alpha x) &= \ell(y) = \ell(X^{\zeta - j \tal}) - \ell(\tw^{-1} s_{\tv \tal} \tv s_{\tal} ) + \ell(\tv s_{\tal}) \\
&= \langle \zeta - j\tal, 2\rho \rangle - \ell(\tw^{-1} \tv) + \ell(\tv s_{\tal}) \\
&= \langle \zeta - j \tal, 2\rho \rangle - \ell(\tv^{-1} \tw) + \ell(\tv s_{\tal}).
\end{align*}

And since $\ell(x) = \langle \zeta, 2\rho \rangle - \ell(\tv^{-1} \tw) + \ell(\tv)$, we have $\ell(x) - \ell(s_\alpha x) = j\langle \tal, 2\rho \rangle + \ell(\tv) - \ell(\tv s_{\tal})$. So $\ell(x) - \ell(s_\alpha x) = 1$ if and only if $1 = j \langle \tal, 2\rho \rangle + \ell(\tv) - \ell(\tv s_{\tal})$.

Using \cite[Prop 6.5]{MM}, which says $\ell(s_{\tal}) \leq \langle \tal, 2\rho \rangle - 1$, and using the fact that $\ell(\tv) - \ell(\tv s_{\tal}) \geq -\ell(s_{\tal})$, we have $1 - j \langle \tal, 2\rho \rangle = \ell(\tv) - \ell(\tv s_{\tal}) \geq 1 - \langle \tal, 2\rho \rangle$. This gives $(1 - j) \langle \tal, 2\rho \rangle \geq 0,$ and since $\tal > 0$ and $j \geq 0,$ we have two possibilities. Either $j = 0$ and $\ell(\tv) = \ell(\tv s_{\tal}) + 1,$ or $j =1$ and $\ell(\tv) = \ell(\tv s_{\tal}) + 1 - \langle \tal , 2\rho \rangle$. Then the form of $y$ is determined by these values of $j$.

Next, suppose $\zeta - (\langle \zeta, \tal \rangle - j)\tal$ is dominant and regular. Using Proposition \ref{prop: 6} and $\ell(s_\alpha x) = \ell(X^{\tv(\zeta - (\langle \zeta, \tal \rangle - j)\tal)} s_{\tv \tal} \tw):$
\begin{align*}
\ell(s_\alpha x) & = \ell(X^{\zeta - (\langle \zeta, \tal \rangle - j)\tal}) - \ell(\tv^{-1} s_{\tv \tal} \tw) + \ell(\tv^{-1}) \\
&= \langle \zeta - (\langle \zeta, \tal \rangle - j)\tal, 2\rho \rangle - \ell(s_{\tal} \tv^{-1} \tw) + \ell(\tv).
\end{align*}

So $1 = \ell(x) - \ell(s_\alpha x) = \ell(s_{\tal} \tv^{-1} \tw) - \ell(\tv^{-1} \tw) - \langle (j - \langle \zeta, \tal \rangle) \tal , 2\rho \rangle$. 

Using \cite[Prop 6.5]{MM} and the fact that $\ell(s_{\tal} \tv^{-1} \tw) - \ell(\tv^{-1} \tw)  \geq -\ell(s_{\tal})$, we have $1 - \langle \tal , 2\rho \rangle \leq \ell(s_{\tal} \tv^{-1} \tw) - \ell(\tv^{-1} \tw) = 1 + (j - \langle \zeta, \tal \rangle) \langle \tal , 2\rho \rangle$, and $0 \leq (j - \langle \zeta, \tal \rangle + 1)\langle \tal ,2\rho \rangle$. 

Using $\langle \tal ,2\rho \rangle > 0$ and $0 \leq j \leq \langle \zeta, \tal \rangle,$ we have two possibilities. Either $j = \langle \zeta, \tal \rangle$ and $\ell(s_{\tal} \tv^{-1} \tw) - \ell(\tv^{-1} \tw) = 1,$ or $j = \langle \zeta, \tal \rangle - 1$ and $\ell(s_{\tal} \tv^{-1} \tw) - \ell(\tv^{-1} \tw) = 1 - \langle \tal, 2\rho \rangle$. Then the form of $y$ is determined by these values of $j$.
\end{proof}

\subsection{Quantum Bruhat Graphs}

\begin{definition}
We define the \textbf{quantum Bruhat graph} (QBG) of $W_{\aff}$ to be the graph whose set of vertices consists of the elements of $W_{\aff}$ and whose edge set is created by making a directed edge from $\tv s_{\tal}$ to $\tv$ for $\tal$ a positive affine root if one of the following holds:

\begin{enumerate}
\item $\ell(\tv) = \ell(\tv s_{\tal} ) + 1$
\item $\ell(\tv) = \ell(\tv s_{\tal} ) - \langle \tal, 2\rho \rangle + 1.$
\end{enumerate}

\noindent The edges are labeled by $\tal$.

\end{definition}

\begin{rem}
The edges in the QBG of $W_{\aff}$ that meet the first length requirement represent covers in the affine Bruhat order of the form $\tv s_{\tal} \lessdot \tv$. They are the edges that appear in the Hasse diagram for $W_{\aff}$. %In \cite[Prop 6.5]{MM}, Mare and Mihalcea partially classified the edges with the second length requirement. 
\end{rem}

\begin{ex}
Let $W_{\aff}$ be of type $\tilde{A}_1$. Then the QBG of $W_{\aff}$ is given below.
\begin{center}
\includegraphics[scale=.5]{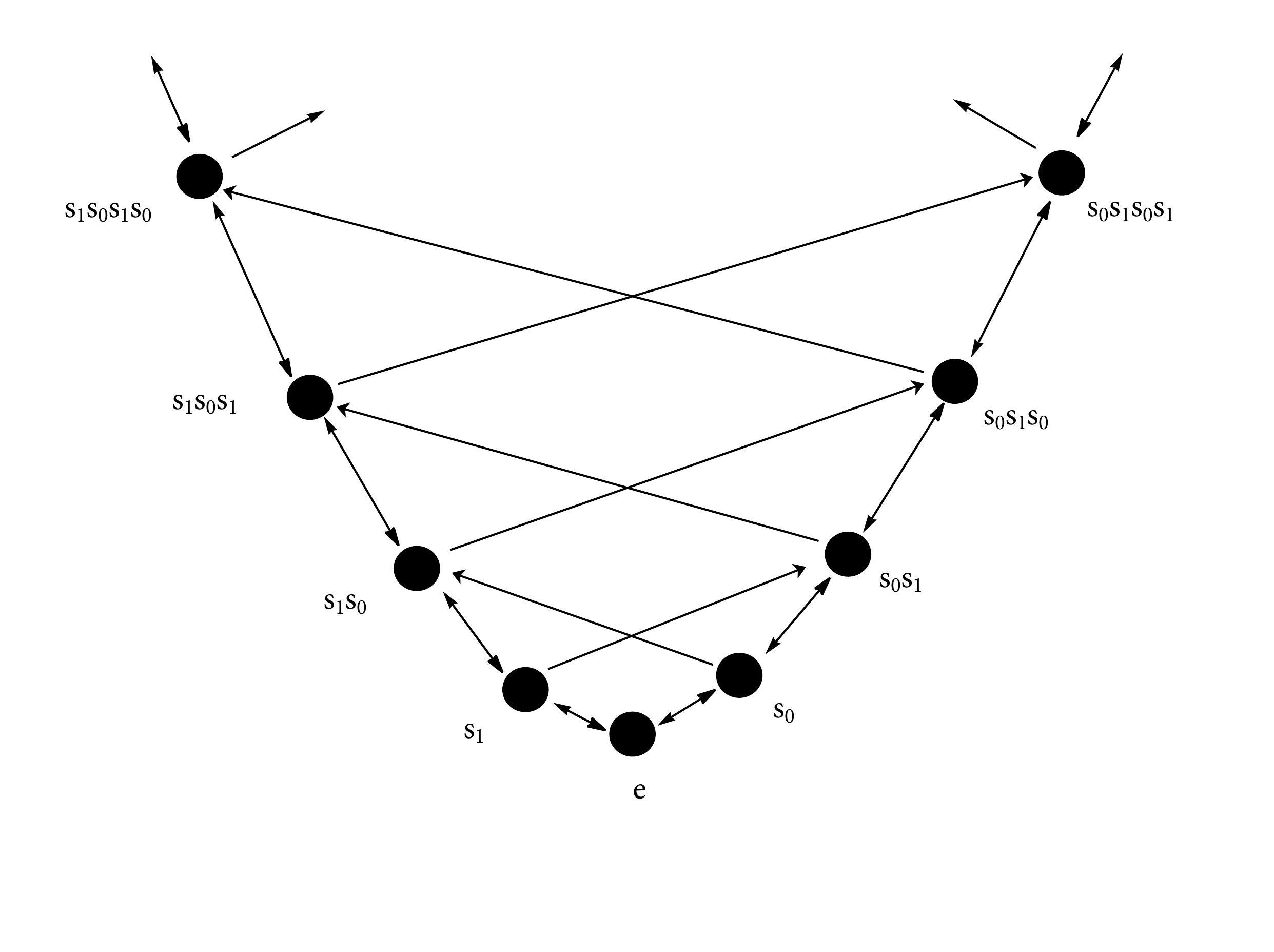}
\end{center}
\end{ex}

\begin{rem}
There is a correspondence from the length conditions required in Theorem \ref{thm: 1} to the edges in the quantum Bruhat graph of $W_{\aff}$.

\begin{itemize}

\item Length condition (1) corresponds to an upward edge in the QBG of the form $\tv s_{\tal} \rightarrow \tv$ with length change +1. 

\item Length condition (2) corresponds to a downward edge in the QBG of the form $\tv s_{\tal} \rightarrow \tv$ with length change $-(\langle \tal, 2\rho \rangle - 1)$. 

\item Length condition (3) corresponds to an upward edge in the QBG of the form $\tw^{-1} \tv \rightarrow \tw^{-1} \tv s_{\tal}$ with length change +1. 

\item Length condition (4) corresponds to a downward edge in the QBG of the form $\tw^{-1} \tv \rightarrow \tw^{-1} \tv s_{\tal}$ with length change $-(\langle \tal, 2\rho \rangle - 1)$. 

\end{itemize}
\end{rem}

Because of the correspondence in Theorem \ref{thm: 1} between the length conditions and the QBG, we can find cocovers in $W$ by considering edges in the QBG of $W_{\aff}$.

\begin{ex}
The QBG with $W_{\aff}$ of type $\tilde{A}_1$ has the upward edge $s_1 s_0 \rightarrow s_0 s_1 s_0$. If we pick $\tv = s_0s_1s_0$ and $\tv s_{\tal} = s_1s_0$, then this edge corresponds to the first cocover type in Theorem \ref{thm: 1} and $j = 0$. The reflection we are extending by is $s_{\tal} = s_0s_1s_0s_1s_0$, so $\tal = s_0s_1(\alpha_0) = 3\alpha_0 + 2\alpha_1 = -\alpha_1 + 3\delta$, and $\alpha = -\tv \tal + j \pi= -\alpha_1 + \delta = \alpha_0$. 

To make the length bound needed in Theorem \ref{thm: 1} small, we pick $\tw = id$. Then $\ell(\tw) = 0$ and $\ell(s_{\tv \tal}\tw) = 1,$ so we can take $M = 1$.

We pick $\zeta = 2\alpha_1 + \delta + 8\Lambda_0$ and check $\langle \zeta, \alpha_i \rangle \geq 2(M + 1) = 4$ for $i = 0, 1$. With these choices,
$$x = X^{s_0 s_1 s_0 (\zeta)} = X^{14\alpha_1 - 23 \delta + 8\Lambda_0}, \ \ y = X^{s_1 s_0 (\zeta)} Y^{\alpha_1} s_1 = X^{-6\alpha_1 - 3\delta + 8\Lambda_0}Y^{\alpha_1}s_1,$$

\noindent and $y$ is a cocover of $x$. Further, we can confirm this by using Sage \cite{S} to check the lengths. Indeed, $\ell(x) -\ell(y) = 8 - 7 = 1$.
\end{ex}

 %%%%%%%%%%%%%%%%%%%%%%%%%%%%%%%%%%%%%%%%%%%%%

\section{Second Method for Classifying Cocovers}

To rid ourselves of the bounds needed on $\ell(\tw)$ and $\ \ell(s_{\tv\tal}\tw)$ in Theorem \ref{thm: 1}, we look at cocovers of $x \in W$ in a different way. We take a more geometrical approach by using the length difference set defined by Muthiah and Orr \cite{MO}.

Recall that we are assuming lev$(x) > 0$ whenever considering $x \in W$, and we are restricting $\Phi_{\fin}$ to be irreducible and simply laced.

 \begin{thm}[MO]
 
Let $x = X^{\zeta}\tw$ with $\zeta \in \mathcal{T}$ and $\tw \in W_{\aff}$. Let $\alpha$ be a positive double affine root such that $x^{-1}(\alpha) < 0$. Then $y =  s_\alpha x \leq x$ with respect to the Bruhat order by definition, and
$$\ell(y) = \ell(x) - |\{\beta \in \Phi^+ \ : \  x^{-1}(\beta) < 0, \ \ s_\alpha(\beta) < 0, \ \ x^{-1}s_{\alpha}(\beta) > 0\}|.$$

\vspace{3mm}

In particular, $L_{x,\alpha} :=\{\beta \in \Phi^+ \ : \  x^{-1}(\beta) < 0, \ \ s_\alpha(\beta) < 0, \ \ x^{-1}s_{\alpha}(\beta) > 0\}$ is finite.
\end{thm}

We call $L_{x,\alpha}$ the length difference set for $x$ and $y = s_\alpha x,$ and note that $y$ is a cocover of $x$ if and only if $L_{x, \alpha} = \{\alpha\}$. This is because $y$ is a cocover if and only if the length difference is 1, and $\alpha$ is always in $L_{x, \alpha}$ if $y =s_\alpha x \leq x$. 

%Comment: \textcolor{orange}{Should I include an explanation of $L_{x, \alpha}$ based on inversions?}

\begin{ex}\label{ex: 1}

Let $W_{\aff}$ be of type $\widetilde{A}_2$, $x = X^{\alpha_1 + \alpha_2+\delta + \Lambda_0}Y^{\alpha_2}$, and $\alpha = \alpha_1 - 2\delta + \pi$. With this setup,
$$L_{x,\alpha} = \{ \alpha, \theta -3\delta + \pi, -\alpha_2 + \delta\}.$$

At this point, we will not go into the detail of checking that these elements do in fact belong to the length difference set (and are the only elements that do belong), but we can do a quick check of the order. Using Sage \cite{S}, we find $\ell(x) = 12$ and $\ell(s_{\alpha}x) = 9,$ so the length difference set must indeed contain 3 elements. 

Note that the two elements of the length difference set that are not $\alpha$ are $\theta - 3\delta + \pi$ and $-\alpha_2 + \delta = -s_{\alpha}(\theta - 3\delta + \pi)$. In general, the elements of the length difference set that are not equal to $\alpha$ will come in such pairs. If $\beta \in L_{x, \alpha}$ and $\beta \neq \alpha$ then $-s_{\alpha}(\beta) \in L_{x, \alpha}$.

\end{ex}

%%%%%%%%%%%%%%%%%%%%%%%%%%%%%%%%%%%%%%%%%%%%%%

\subsection{Graphs}

To take a geometrical approach, we need a way to envision the elements of the length difference set. We will begin by graphing the positive double affine roots $\alpha$ such that $y = s_{\alpha}x$ is less than $x$ with respect to the Bruhat order. 

\begin{definition}
Let $\nu \in \Phi_{\fin}$ and let $\Gamma_{x,\nu}$ denote the points $(r,j) \in \mathbb{Z}^2$ such that $\alpha = \nu + r\delta + j\pi > 0$ and $x^{-1}(\alpha) < 0$. We call this the \textbf{lower graph of $x$ corresponding to $\nu$} and say $\alpha$ corresponds to a point in $\Gamma_{x,\nu}$ if $\alpha = \nu + r\delta + j\pi$ such that $(r,j) \in \Gamma_{x, \nu}$.
\end{definition}

\begin{figure}[H]
\begin{tikzpicture}[scale = .8]
%shading:
\fill[blue!50!cyan, opacity=0.3] (-4,4) -- (0,0) -- (-6,0) -- (-6,4) -- cycle;
%lines:
\draw [->, ultra thick] (0,0) -- (-4,4);
\draw [->, ultra thick] (0,0) -- (-6,0);
%points:
%\draw[fill = black] (0,0) circle [radius = 0.1];
%corners:
%\draw [red, ultra thick] (0,0) circle [radius=0.1];
\end{tikzpicture}
\caption{A general $\Gamma_{x, \nu}$} \label{fig:cocovers}
\end{figure}
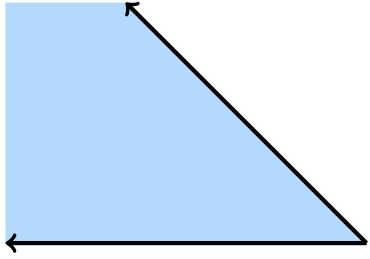

It is important to note that the two outer edges appearing in the graph above may or may not be included (and it is very possible that only part of an edge will be included) depending on the choice of $x$ and $\nu$. 

Because the graph shows all $\alpha$ such that $x \geq s_{\alpha}x$, it is clear that the order of $L_{x,\alpha}$ will be greater than or equal to one if $\alpha$ corresponds to a point in $\Gamma_{x,\nu}.$

\begin{definition}
Let $\alpha = \nu + r\delta + j\pi$ be a double affine root. Then we say $\nu$ is the \textbf{finite part of} $\alpha$ because $\nu \in \Phi_{\fin}$. We denote this by $\fin(\alpha) = \nu$.
\end{definition}

\begin{prop}\label{prop: 7}
Let $\alpha = \nu + r\delta + j\pi$ and $\beta = \gamma + p \delta + q\pi$. The double affine root $\beta$ is in $L_{x,\alpha}$ if and only if $\beta \in \Gamma_{x, \fin(\beta)}$ and $-s_{\alpha}\beta \in \Gamma_{x, \fin(-s_{\alpha}(\beta))}$.
\end{prop}

\begin{proof}
Note that $\fin(\alpha) = \nu,$ $\fin(\beta) = \gamma,$ and $\fin(-s_{\alpha}(\beta)) = -s_{\nu}(\gamma)$.

Let $\beta \in L_{x,\alpha}$. Then $\beta > 0$ and $x^{-1}(\beta) < 0$, so $\beta \in \Gamma_{x, \gamma}$. Additionally, $s_\alpha(\beta) < 0$ and $x^{-1}(s_\alpha \beta) > 0$, so $-s_{\alpha}(\beta) \in \Gamma_{x, -s_\nu(\gamma)}.$

Let $\beta \in \Gamma_{x, \gamma}$ and $-s_{\alpha}(\beta) \in \Gamma_{x, -s_\nu(\gamma)}.$ Then $\beta > 0$, $x^{-1}(\beta) < 0$, $-s_\alpha(\beta) > 0,$ and $-x^{-1}(s_\alpha(\beta)) < 0$, so $\beta \in L_{x,\alpha}$. 
\end{proof}

 \begin{ex}
  
Consider $W_{\aff}$ of type $\widetilde{A}_2$ and $x = X^{\alpha_1 + \alpha_2+\delta + \Lambda_0}Y^{\alpha_2}$ (the same choices from Example \ref{ex: 1}). The lower graph of $x$ corresponding to $\alpha_1$ is given below.
  
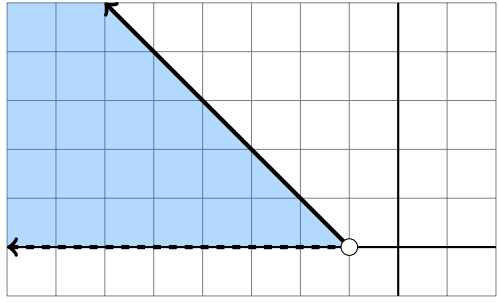
\begin{figure}[H]
\begin{tikzpicture}[scale = .65]
%grid:
\draw[help lines] (-8,-1) grid (2,5);
%shading:
\fill[blue!50!cyan, opacity=0.3] (-8,0) -- (-1,0) -- (-6,5) -- (-8, 5) -- cycle;
%lines:
\draw [->, ultra thick] (-1,0) -- (-6,5);
\draw [<-, ultra thick] [dashed] (-8,0) -- (-1,0);
\draw[thick](-8,0)--(2,0);
\draw[thick](0,5)--(0,-1);
%points:
\draw[fill=white, radius = .1] (-1,0) circle [radius = 0.17];
\end{tikzpicture}
\caption{$\Gamma_{x, \alpha_1}$} \label{fig:alp1}
\end{figure}
To see why this is the graph for $\Gamma_{x, \alpha_1},$ we need to examine when $\alpha > 0$ and $x^{-1}(\alpha) < 0$. 

The double affine root $\alpha = \alpha_1 + r\delta + j\pi$ is positive if and only if one of the following holds:

\begin{enumerate}
\item $j > 0$
\item $j = 0$ and $r > 0$.
\end{enumerate}

To determine when $x^{-1}(\alpha) < 0$, it will help to expand $x^{-1}(\alpha)$:
\begin{align*}
x^{-1}(\alpha) & = Y^{-\alpha_2}X^{-\alpha_1 - \alpha_2 - \delta - \Lambda_0}(\alpha_1 + r\delta + j\pi)\\ 
& = \alpha_1 + (r + \langle \alpha_1, \alpha_2 \rangle)\delta + (j + \langle \alpha_1 + r \delta, -\alpha_1 - \alpha_2 - \delta - \Lambda_0 \rangle) \pi \\
& = \alpha_1 + (r - 1) \delta + (j + r + 1) \pi.
\end{align*}

Now we can see that $x^{-1}(\alpha) < 0$ if and only if one of the following holds:
\begin{enumerate}
\item $j < -r - 1$
\item $j = -r - 1$ and $r < 1.$
\end{enumerate}

Combining these restrictions results in the graph shown above.
\end{ex}

\begin{prop}
Fix $x = X^{\zeta}\tw \in W$ with $\tw = Y^{\lambda}w \in W_{\aff}$ and fix $\nu \in \Phi_{\fin}$. The point $(r, j) \in \Gamma_{x, \nu}$ if and only if one of the following holds:
\begin{enumerate}
\item $0 < j <  \langle - \zeta, \tal \rangle = - \langle \zeta, \nu + r\delta \rangle$
\item $(r, j) = (r, 0)$ with $0 \leq r \leq \frac{\langle \nu, \mu \rangle}{- l}$,  and if $r = 0$, then $\nu > 0$
\item $(r, j) = (r, \langle - \zeta, \tal \rangle) = (r, - \langle \nu, \mu \rangle - lr)$ with $r \leq \ $min$\{ \frac{\langle \nu, \mu \rangle}{- l}, - \langle \lambda, \nu \rangle \}$, and if $r = - \langle \lambda, \nu \rangle$, then $w^{-1}(\nu) < 0$.
\end{enumerate}
\end{prop}

\begin{proof}
For $\alpha = \nu + r\delta + j\pi \in \Phi$ to correspond to a point in $\Gamma_{x, \nu}$, we need both $\alpha > 0$ and $x^{-1}(\alpha)<0$.

For $\alpha = \nu + r\delta + j\pi > 0,$ we need one of the following:

\begin{enumerate}
\item $j > 0$
\item $j = 0,\ r > 0$
\item $j = 0 , \  r = 0, \nu > 0.$
\end{enumerate}

For $x^{-1}(\alpha) < 0,$ we need 
\begin{align*}
x^{-1}(\alpha) = & \ \tw^{-1}(\tal) + (j - \langle - \zeta, \tal \rangle ) \pi\\
= & \ w^{-1}(\nu) + (r + \langle \lambda, \nu \rangle) \delta + (j + \langle \mu, \nu \rangle + l r)\pi < 0,
\end{align*}

so we need one of the following:

\begin{enumerate}
\item $j < \langle - \zeta, \tal \rangle = - \langle \mu, \nu \rangle - l r$
\item $ j = \langle - \zeta, \tal \rangle,  \ r < - \langle \lambda, \nu \rangle$
\item $ j = \langle - \zeta, \tal \rangle, \ r = - \langle \lambda, \nu \rangle, w^{-1}(\nu) < 0.$

\end{enumerate}

Combining these results, we see that if $(r, j)$ is in the graph, then $0 \leq j \leq \langle - \zeta, \tal \rangle$. This tells us that $-\langle \zeta, \tal \rangle \geq0$ and since $-\langle \zeta, \tal \rangle = \langle -\mu -m\delta -l \Lambda_0, \nu + r\delta \rangle = -\langle \mu, \nu \rangle - lr$, we can solve for $r$ and get $r \leq \frac{\langle \nu, \mu \rangle}{- l}$. The point $(r, j) = ( \frac{\langle \nu, \mu \rangle}{- l}, 0)$ is the intersection point of $j = 0$ and $j = \langle - \zeta, \tal \rangle$. When $j = 0$, we can restrict $r$ to $0 \leq r \leq \frac{\langle \nu, \mu \rangle}{- l}$, and when $j = \langle - \zeta, \tal \rangle$, we can restrict to $r \leq \ $min$\{ \frac{\langle \nu, \mu \rangle}{- l}, - \langle \lambda, \nu \rangle \}$. 
\end{proof}

\begin{definition}
For simplicity, we will refer to the line segment of $j = 0$ that is included in the graph and the ray of $j = - \langle \zeta, \tal \rangle = - \langle \zeta, \nu + r\delta \rangle$ that is included in the graph as the \textbf{lower and upper outer edges} respectively. We will refer to the ray of $j = 1$ that is included in the graph and the ray of $j = - \langle \zeta, \nu + r\delta \rangle - 1$ that is included in the graph as the \textbf{lower and upper inner edges} respectively.
\end{definition}

\begin{prop} For a fixed $x \in W$ and $\nu \in \Phi_{\fin}$, there are 12 possible forms for $\Gamma_{x, \nu}$, and they are represented by the graphs below.\\

\begin{center} 
\begin{tikzpicture}[scale = .3]
%shading:
\fill[blue!50!cyan, opacity=0.3] (-6,0) -- (0,0) -- (-6,4) -- cycle;
%lines:
\draw [->, ultra thick] (0,0) -- (-6,4);
\draw [<-, ultra thick] [dashed] (-6,0) -- (0,0);
%points:
\draw[fill=white, radius = 1] (0,0) circle [radius = 0.25];
\end{tikzpicture}
\qquad
\begin{tikzpicture}[scale = .3]
%shading:
\fill[blue!50!cyan, opacity=0.3] (-6,0) -- (0,0) -- (-6,4) -- cycle;
%lines:
\draw [->, ultra thick] (-3,2) -- (-6,4);
\draw [ultra thick, dashed] (0,0) -- (-3, 2);
\draw [<-, ultra thick] [dashed] (-6,0) -- (0,0);
%points:
\draw[fill=white, radius = 1] (0,0) circle [radius = 0.25];
\draw[fill=white, radius = 1] (-3,2) circle [radius = 0.25];
\end{tikzpicture}
\qquad
\begin{tikzpicture}[scale = .3]
%shading:
\fill[blue!50!cyan, opacity=0.3] (-6,0) -- (0,0) -- (-6,4) -- cycle;
%lines:
\draw [->, ultra thick] (-3,2) -- (-6,4);
\draw [ultra thick, dashed] (0,0) -- (-3, 2);
\draw [<-, ultra thick] [dashed] (-6,0) -- (0,0);
%points:
\draw[fill=white, radius = 1] (0,0) circle [radius = 0.25];
\draw[fill=black, radius = 1] (-3,2) circle [radius = 0.25];
\end{tikzpicture}
\qquad
\begin{tikzpicture}[scale = .3]
%shading:
\fill[blue!50!cyan, opacity=0.3] (-6,0) -- (0,0) -- (-6,4) -- cycle;
%lines:
\draw [->, ultra thick] (0,0) -- (-6,4);
\draw [<-, ultra thick] [dashed] (-6,0) -- (-3,0);
\draw [ultra thick] (0,0) -- (-3, 0);
%points:
\draw[fill=white, radius = 1] (0,0) circle [radius = 0.25];
\draw[fill=white, radius = 1] (-3,0) circle [radius = 0.25];
\end{tikzpicture}

\vspace{3mm}

\begin{tikzpicture}[scale = .3]
%shading:
\fill[blue!50!cyan, opacity=0.3] (-6,0) -- (0,0) -- (-6,4) -- cycle;
%lines:
\draw [->, ultra thick] (-3,2) -- (-6,4);
\draw [ultra thick, dashed] (0,0) -- (-3, 2);
\draw [<-, ultra thick] [dashed] (-6,0) -- (-3,0);
\draw [ultra thick] (0,0) -- (-3, 0);
%points:
\draw[fill=white, radius = 1] (0,0) circle [radius = 0.25];
\draw[fill=white, radius = 1] (-3,2) circle [radius = 0.25];
\draw[fill=white, radius = 1] (-3,0) circle [radius = 0.25];
\end{tikzpicture}
\qquad
\begin{tikzpicture}[scale = .3]
%shading:
\fill[blue!50!cyan, opacity=0.3] (-6,0) -- (0,0) -- (-6,4) -- cycle;
%lines:
\draw [->, ultra thick] (-3,2) -- (-6,4);
\draw [ultra thick, dashed] (0,0) -- (-3, 2);
\draw [<-, ultra thick] [dashed] (-6,0) -- (-3,0);
\draw [ultra thick] (0,0) -- (-3, 0);
%points:
\draw[fill=white, radius = 1] (0,0) circle [radius = 0.25];
\draw[fill=black, radius = 1] (-3,2) circle [radius = 0.25];
\draw[fill=white, radius = 1] (-3,0) circle [radius = 0.25];
\end{tikzpicture}
\qquad
\begin{tikzpicture}[scale = .3]
%shading:
\fill[blue!50!cyan, opacity=0.3] (-6,0) -- (0,0) -- (-6,4) -- cycle;
%lines:
\draw [->, ultra thick] (0,0) -- (-6,4);
\draw [<-, ultra thick] [dashed] (-6,0) -- (-3,0);
\draw [ultra thick] (0,0) -- (-3, 0);
%points:
\draw[fill=white, radius = 1] (0,0) circle [radius = 0.25];
\draw[fill=black, radius = 1] (-3,0) circle [radius = 0.25];
\end{tikzpicture}
\qquad
\begin{tikzpicture}[scale = .3]
%shading:
\fill[blue!50!cyan, opacity=0.3] (-6,0) -- (0,0) -- (-6,4) -- cycle;
%lines:
\draw [->, ultra thick] (-3,2) -- (-6,4);
\draw [ultra thick, dashed] (0,0) -- (-3, 2);
\draw [<-, ultra thick] [dashed] (-6,0) -- (-3,0);
\draw [ultra thick] (0,0) -- (-3, 0);
%points:
\draw[fill=white, radius = 1] (0,0) circle [radius = 0.25];
\draw[fill=white, radius = 1] (-3,2) circle [radius = 0.25];
\draw[fill=black, radius = 1] (-3,0) circle [radius = 0.25];
\end{tikzpicture}

\vspace{3mm}

\begin{tikzpicture}[scale = .3]
%shading:
\fill[blue!50!cyan, opacity=0.3] (-6,0) -- (0,0) -- (-6,4) -- cycle;
%lines:
\draw [->, ultra thick] (-3,2) -- (-6,4);
\draw [ultra thick, dashed] (0,0) -- (-3, 2);
\draw [<-, ultra thick] [dashed] (-6,0) -- (-3,0);
\draw [ultra thick] (0,0) -- (-3, 0);
%points:
\draw[fill=white, radius = 1] (0,0) circle [radius = 0.25];
\draw[fill=black, radius = 1] (-3,2) circle [radius = 0.25];
\draw[fill=black, radius = 1] (-3,0) circle [radius = 0.25];
\end{tikzpicture}
\qquad
\begin{tikzpicture}[scale = .3]
%shading:
\fill[blue!50!cyan, opacity=0.3] (-6,0) -- (0,0) -- (-6,4) -- cycle;
%lines:
\draw [->, ultra thick] (0,0) -- (-6,4);
\draw [<-, ultra thick] [dashed] (-6,0) -- (-3,0);
\draw [ultra thick] (0,0) -- (-3, 0);
%points:
\draw[fill=black, radius = 1] (0,0) circle [radius = 0.25];
\draw[fill=white, radius = 1] (-3,0) circle [radius = 0.25];
\end{tikzpicture}
\qquad
\begin{tikzpicture}[scale = .3]
%shading:
\fill[blue!50!cyan, opacity=0.3] (-6,0) -- (0,0) -- (-6,4) -- cycle;
%lines:
\draw [->, ultra thick] (0,0) -- (-6,4);
\draw [<-, ultra thick] [dashed] (-6,0) -- (0,0);
%points:
\draw[fill=black, radius = 1] (0,0) circle [radius = 0.25];
\end{tikzpicture}
\qquad
\begin{tikzpicture}[scale = .3]
%shading:
\fill[blue!50!cyan, opacity=0.3] (-6,0) -- (0,0) -- (-6,4) -- cycle;
%lines:
\draw [->, ultra thick] (0,0) -- (-6,4);
\draw [<-, ultra thick] [dashed] (-6,0) -- (-3,0);
\draw [ultra thick] (0,0) -- (-3, 0);
%points:
\draw[fill=black, radius = 1] (0,0) circle [radius = 0.25];
\draw[fill=black, radius = 1] (-3,0) circle [radius = 0.25];
\end{tikzpicture}
\end{center}

\end{prop}

\begin{proof}
Because the lower outer edge will be the line segment $j = 0$ with $0 \leq r \leq \frac{\langle \nu, \mu \rangle}{- l}$ and endpoints possibly not included, there are four possibilities:

\begin{figure}[H]
\begin{tikzpicture}[scale = .35]
%lines:
\draw [ultra thick] (0,0) -- (6,0) node[below, yshift=-3mm, xshift = -10mm] {L1};
%points:
\draw[fill=white, radius = 1] (0,0) circle [radius = 0.2];
\draw[fill=white, radius = .5] (6,0) circle [radius = 0.2];
\end{tikzpicture}
\qquad
\begin{tikzpicture}[scale = .35]
%lines:
\draw [ultra thick] (0,0) -- (6,0) node[below, yshift=-3mm, xshift = -10mm] {L2};
%points:
\draw[fill=black, radius = 1] (0,0) circle [radius = 0.2];
\draw[fill=white, radius = .5] (6,0) circle [radius = 0.2];
\end{tikzpicture}
\qquad
\begin{tikzpicture}[scale = .35]
%lines:
\draw [ultra thick] (0,0) -- (6,0) node[below, yshift=-3mm, xshift = -10mm] {L3};
%points:
\draw[fill=white, radius = 1] (0,0) circle [radius = 0.2];
\draw[fill=black, radius = .5] (6,0) circle [radius = 0.2];
\end{tikzpicture}
\qquad
\begin{tikzpicture}[scale = .35]
%lines:
\draw [ultra thick] (0,0) -- (6,0) node[below, yshift=-3mm, xshift = -10mm] {L4};
%points:
\draw[fill=black, radius = 1] (0,0) circle [radius = 0.2];
\draw[fill=black, radius = .5] (6,0) circle [radius = 0.2];
\end{tikzpicture}
\qquad
\end{figure}

Note that the first two types do not include the intersection point of $j = 0$ and $j = - \langle \zeta, \tal \rangle$ (represented by the right endpoint), but the last two types do. Also note that type L1 and type L4 may or may not include the line segment between the endpoints. If the line segment is not included, we will refer to these as type L1* or L4* respectively.

Now we will look at the possibilities for the upper outer edge given by $j = - \langle \zeta, \tal \rangle$ with $r \leq \ $min$\{ \frac{\langle \nu, \mu \rangle}{- l}, - \langle \lambda, \nu \rangle \}$ and endpoint possibly not included:

\begin{figure}[H]
\begin{tikzpicture}[scale = .35]
%lines:
\draw [ultra thick, dashed] (3,0) -- (6,0) node[below, yshift=-3mm, xshift = -10mm] {U1};
\draw [<-, ultra thick] (0,0) -- (3,0);
%points:
\draw[fill=white, radius = 1] (3,0) circle [radius = 0.2];
\draw[fill=white, radius = .5] (6,0) circle [radius = 0.2];
\end{tikzpicture}
\qquad
\begin{tikzpicture}[scale = .35]
%lines:
\draw [ultra thick, dashed] (3,0) -- (6,0) node[below, yshift=-3mm, xshift = -10mm] {U2};
\draw [<-, ultra thick] (0,0) -- (3,0);
%points:
\draw[fill=black, radius = 1] (3,0) circle [radius = 0.2];
\draw[fill=white, radius = .5] (6,0) circle [radius = 0.2];
\end{tikzpicture}
\qquad
\begin{tikzpicture}[scale = .35]
%lines:
\draw [<-, ultra thick] (0,0) -- (6,0) node[below, yshift=-3mm, xshift = -10mm] {U3};
%points:
\draw[fill=white, radius = .5] (6,0) circle [radius = 0.2];
\end{tikzpicture}
\qquad
\begin{tikzpicture}[scale = .35]
%lines:
\draw [<-, ultra thick] (0,0) -- (6,0) node[below, yshift=-3mm, xshift = -10mm] {U4};
%points:
\draw[fill=black, radius = .5] (6,0) circle [radius = 0.2];
\end{tikzpicture}
\end{figure}

Note that the first three types do not include the intersection point of $j = 0$ and $j = - \langle \zeta, \tal \rangle$ (represented by the right point). These upper outer edges will match with the lower outer edges of type L1, L1*, and L2. The only upper outer edge containing the intersection point is of type U4, so this will match with the lower outer edges of type L3, L4, and L4*. In total, this gives 12 possibilities for the graph.
\end{proof}

%%%%%%%%%%%%%%%%%%%%%%%%%%%%%%%%%%%%%%%%%%%%%%

\subsection{Corners}

Recall that we are interested in determining which $\alpha$ of $\Gamma_{x,\fin(\alpha)}$ correspond to cocovers (meaning $y = s_\alpha x$ is a cocover of $x$). To do this, we must examine specific $(r, j) \in \Gamma_{x, \fin(\alpha)}$.

\begin{definition}
For double affine roots $\alpha = \nu + r\delta + j\pi$ and $\beta = \nu + p \delta + q\pi$, define $\beta_{\alpha}^{-}$ to be the root found by rotating $(p,q)$ 180 degrees about $(r,j)$. 
\end{definition}

\begin{prop}
If $\beta$ and $\alpha$ are double affine roots such that $\fin(\alpha) = \fin(\beta)$, then $\beta_{\alpha}^{-} = -s_{\alpha}\beta$.
\end{prop}

\begin{proof}
Let $\alpha = \nu + r\delta + j\pi$ and $\beta = \nu + p\delta + j\pi$. Then 
\begin{align*}
-s_{\alpha}(\beta) & = -\beta + \langle \beta, \alpha \rangle \alpha\\
& = -\beta + \langle \nu, \nu \rangle \alpha\\
& = -\beta + 2 \alpha\\
& = \nu + (2r - p)\delta + (2j - q)\pi.
\end{align*}

The root $\beta_{\alpha}^{-}$ is equal to $\nu + p'\delta + q'\pi$ where $(p',q')$ is the result of rotating $(p, q)$ 180 degrees about $(r, j)$. To determine $(p', q')$, first shift so that we are rotating about the center: $(p, q) \rightarrow (p-r, q-j)$ and $(r, j) \rightarrow (0, 0)$, then reflect over the $x$ and $y$ axes: $(p-r, q-j) \rightarrow (-p + r, -q + j)$, and now shift back to original orientation: $(0, 0) \rightarrow (r, j)$ and $(-p + r, -q + j) \rightarrow (-p +2r, -q + 2j) = (2r - p, 2j - q)$.

So $(p', q') = (2r - p, 2j - q)$ and $\beta_{\alpha}^{-} = \nu + (2r - p)\delta + (2j - q) \pi = -s_{\alpha}(\beta).$
\end{proof}
  
\begin{definition}
We say that $\alpha$ is a \textbf{corner of the graph }$\Gamma_{x,\fin(\alpha)}$, or a \textbf{corner relative to $x$}, if $\alpha$ corresponds to a point in $\Gamma_{x, \fin(\alpha)},$ and if for any $\beta = \fin(\alpha) + p\delta + q\pi$ corresponding to a point in $\Gamma_{x, \fin(\alpha)}$, $\beta_{\alpha}^{-}$ is not in the graph.
\end{definition}

 \begin{ex}\label{ex: 2}
  
Consider $W_{\aff}$ of type $\widetilde{A}_2,$ $x = X^{\alpha_1 + \alpha_2+\delta + \Lambda_0}Y^{\alpha_2}$ (the same choices from Example \ref{ex: 1}).
  
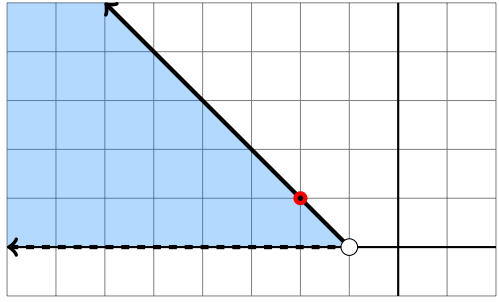
\begin{figure}[H]
\begin{tikzpicture}[scale = .65]
%grid:
\draw[help lines] (-8,-1) grid (2,5);
%shading:
\fill[blue!50!cyan, opacity=0.3] (-8,0) -- (-1,0) -- (-6,5) -- (-8, 5) -- cycle;
%lines:
\draw [->, ultra thick] (-1,0) -- (-6,5);
\draw [<-, ultra thick] [dashed] (-8,0) -- (-1,0);
\draw[thick](-8,0)--(2,0);
\draw[thick](0,5)--(0,-1);
%points:
\draw[fill=white, radius = .1] (-1,0) circle [radius = 0.17];
\draw[fill=black, radius = .08] (-2,1) circle [radius = 0.1];
%corners:
\draw [red, ultra thick] (-2,1) circle [radius=0.1];
\end{tikzpicture}
\caption{$\Gamma_{x, \alpha_1}$ with highlighted corner} \label{fig:alp1}
\end{figure}

With this setup, $\alpha = \alpha_1 - 2\delta + \pi$ corresponds to a corner of $\Gamma_{x,\alpha_1}.$

\end{ex}

\begin{prop}
If $y = s_{\alpha}x$ is a cocover of $x$, then $\alpha$ must correspond to a corner in the graph $\Gamma_{x, \fin(\alpha)}$.
\end{prop}

\begin{proof}
Suppose $\alpha$ is not a corner of $\Gamma_{x, \fin(\alpha)}$. Then there is some $\beta$ with $\fin(\beta) = \fin(\alpha)$ such that $\beta \neq \alpha$, $\beta \in \Gamma_{x, \fin(\alpha)},$ and $\beta_{\alpha}^- \in \Gamma_{x, \fin(\alpha)}$. But $\beta_{\alpha}^- = -s_{\alpha}(\beta)$, so by Proposition \ref{prop: 7},
$\beta \in L_{x, \alpha}$. So $|L_{x,\alpha}| > 1$, and $y$ is not a cocover of $x$. 
\end{proof}

\begin{rem}
In general, the set of corners will be larger than the set of roots corresponding to cocovers of a fixed $x \in W$. Consider $W_{\aff} $ of type $\widetilde{A}_2,$ $x = X^{\alpha_1 + \alpha_2+\delta + \Lambda_0}Y^{\alpha_2}$, $\alpha = \alpha_1 - 2\delta + \pi$. Then $\alpha$ corresponds to a corner of $\Gamma_{x, \alpha_1},$ as shown in Example \ref{ex: 2}, but $s_{\alpha}x$ is not a cocover of $x$ (we saw in Example \ref{ex: 1} that the length difference set contains 3 elements).
\end{rem}

Now we would like to show that there are finitely many $\alpha$ that are corners relative to a fixed $x$, but before we do, we need to make some observations about the graphs.

Fix $\nu \in \Phi_{\fin}$ and $x \in W$. Then:

\begin{itemize}

\item If $(r, j)$ is a point of $\Gamma_{x, \nu}$ and $j \neq 0$, then $(p, j)$ is in $\Gamma_{x, \nu}$ for all $p < r$. This is true because if $(r, j)$ is a point in $\Gamma_{x, \nu}$ with $j \neq 0$, then the only possible bound on $r$ is the upper bound $r \leq \ $min$\{ \frac{\langle \nu, \mu \rangle}{- l}, - \langle \lambda, \nu \rangle \}$. So if $p < r$, then $p <$ min$\{ \frac{\langle \nu, \mu \rangle}{- l}, - \langle \lambda, \nu \rangle \}$ and $(p, j)$ is in $\Gamma_{x, \nu}$. 

\item The upper outer edge falls on the line $y = -\langle \zeta, \nu + x\delta \rangle = - \langle \mu, \nu \rangle - x l$. The slope of this line is $-l$, which is an integer (specifically, $l = $lev$(\zeta)$), and for any $r \in \Z$, $j = - \langle \zeta, \nu + r\delta \rangle$ is also an integer because $- \langle \zeta, \nu + r\delta \rangle = - \langle \mu, \nu \rangle - r l$ where $r, l, \langle \mu, \nu \rangle \in \mathbb{Z}$.

\item Let $L_0$ represent the line given by $y = -\langle \zeta, \nu + x\delta \rangle$. Let $(r, j)$ be a point of $L_0$ that falls in $\Gamma_{x, \nu}$. Then $r \leq \ $min$\{ \frac{\langle \nu, \mu \rangle}{- l}, - \langle \lambda, \nu \rangle \}$. Additionally, for any $(p, q)$ of $L_0$ with $p < r$, $(p, q)$ is in $\Gamma_{x, \nu}$. This is true because $(p, q)$ satisfies $q = -\langle \zeta, \nu + p\delta \rangle$ and $p < r  \leq \ $min$\{ \frac{\langle \nu, \mu \rangle}{- l}, - \langle \lambda, \nu \rangle \}$.

\item Let $L_k$ represent the line given by $y = -\langle \zeta, \nu + x\delta \rangle - k$, where $k$ is a positive integer. Let $(r, j)$ be a point on $L_k$ that falls in the graph. Then $j \geq 0$, and for any $(p, q)$ of $L_k$ with $p < r$, $(p, q)$ is in $\Gamma_{x, \nu}$. This is true because $p < r$ means $q > j$ (since the slope of $L_k$ is negative), so $0 \leq j < q < -\langle \zeta, \nu + p\delta \rangle,$ and $(p, q)$ is a point of $\Gamma_{x,\nu}$.

\end{itemize}

\begin{prop}
Fix $x \in W$ and $\nu \in \Phi_\fin$. The number of corners of $\Gamma_{x, \nu}$ is finite.
\end{prop}

Idea: We show that if $\alpha = \nu + r\delta + j\pi$ corresponds to a corner relative to $x$, then $\alpha$ must fall on one of the two outer edges or one of the two inner edges of $\Gamma_{x, \nu}$. But on these edges, only the $(r, j) \in \Z^2$ closest to endpoints can be corners.

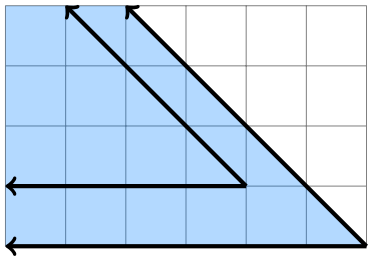
\begin{figure}[H]
\begin{tikzpicture}[scale = .8]
\draw[help lines] (-6,-0) grid (0,4);
%shading:
\fill[blue!50!cyan, opacity=0.3] (-4,4) -- (0,0) -- (-6,0) -- (-6,4) -- cycle;
%lines:
\draw [->, ultra thick] (0,0) -- (-4,4);
\draw [->, ultra thick] (0,0) -- (-6,0);
\draw [->, ultra thick] (-2,1) -- (-6,1);
\draw [->, ultra thick] (-2,1) -- (-5,4);
%points:
%\draw[fill = black] (0,0) circle [radius = 0.1];
%corners:
%\draw [red, ultra thick] (0,0) circle [radius=0.1];
\end{tikzpicture}
\caption{Inner and Outer Edges of a General $\Gamma_{x, \nu}$} \label{fig:cocovers}
\end{figure}

\begin{proof}
We break the proof into several cases. Let $(r',j')$ represent a corner.

Case 1: Assume $j' = 0$. Then $(r', j')$ falls along the lower outer edge, which is either a line segment or a single point. In either case, there are finitely many possibilities for $(r', j')$. 

Case 2: Assume $(r', j')$ falls along the upper outer edge (the diagonal $j = - \langle \zeta, \nu + r\delta \rangle$). Then any other $(r, j) \in \Gamma_{x,\nu}$ on the outer upper edge must have $r < r'$. If there exists some $(r,j)$ on the graph's upper outer edge such that $r > r'$, then it can be rotated 180 degrees about $(r', j')$ and end up in the graph (because it will land on the diagonal and be higher up than $(r', j')$), which contradicts the fact that $(r', j')$ is a corner. So there is only one possibility for $(r', j')$. 

Case 3: Assume $(r', j')$ falls along the upper inner edge (the diagonal given by $j = - \langle \zeta, \nu + r\delta \rangle - 1$). Then using the same logic from above, $(r', j')$ must have largest possible $r',$ so there is only one possibility for $(r', j')$.

Case 4: Assume $j' = 1$. Again, $r'$ must be maximal. Suppose $(r, 1)$ is another point of the graph such that $r > r'$. Then $(r,1)$ rotated 180 degrees about $(r', 1)$ results in some $(p, 1)$ with $p < r'$. This would mean that $(p, 1)$ is in the graph, but that contradicts the fact that $(r', j')$ is a corner.

Case 5: Assume $(r', j')$ does not lie on any of the outer or inner edges. Then $1 < j' < \langle -\mu, \nu \rangle - r'l - 1$, so $1 \leq j' - 1 < j' < \langle -\mu, \nu \rangle - r'l - 1$, and $(r', j' - 1)$ is a point of the graph. Additionally, $1 < j'  < j' + 1 \leq \langle -\mu, \nu \rangle - r'l -1$, so $(r', j' + 1)$ is also a point on the graph. Thus $(r', j')$ cannot be a corner. 

To be a corner, $(r',j')$ must fall along one of the two outer edges or one of the two inner edges. On those edges there are finitely many possibilities for corners. Thus for any given $x$ and $\nu,$ the corresponding graph $\Gamma_{x, \nu}$ contains finitely many corners. 
\end{proof}

\begin{cor}
The number of cocovers of $x$ is finite. 
\end{cor}

\begin{proof}
Fix $x$. For any $\nu \in \Phi_{\fin},$ the graph $\Gamma_{x, \nu}$ has finitely many corners. So there are finitely many $\alpha = \nu + r\delta + j \pi$ such that $y = s_{\alpha} x$ is a cocover of $x$. Since $\nu$ is a finite root, there are finitely many possibilities for $\nu$. So there are finitely many cocovers for a given $x$. 
\end{proof}

\begin{cor}\label{fin} 
Let $x, y \in W$ such that $y \leq x$. Then the double affine Bruhat interval $[y, x]$ will be finite. 
\end{cor}

\begin{proof}
Consider a path from $y$ to $x$. The path has at most $\ell(x) - \ell(y)$ steps. Assume the path is a saturated chain. Then there are finitely many options for each step (because there are finitely many covers for any element of $W$), so there are finitely many options for a saturated chain. And since every $z$ such that $y < z < x$ is part of a saturated chain, we see there are finitely many options for $z \in [y, x]$ and $[y, x]$ must be finite.
\end{proof}

\begin{ex}
Consider $W_{\aff}$ of type $\widetilde{A}_2,$ and let $x = X^{\alpha_1 + \alpha_2+\delta + \Lambda_0}Y^{\alpha_2}$ and $\alpha = \alpha_1 - 2\delta + \pi$ (the same choices from Example \ref{ex: 1}). Then $[s_{\alpha}x, x]$ contains 8 elements:
\begin{center}
\includegraphics[scale=.65]{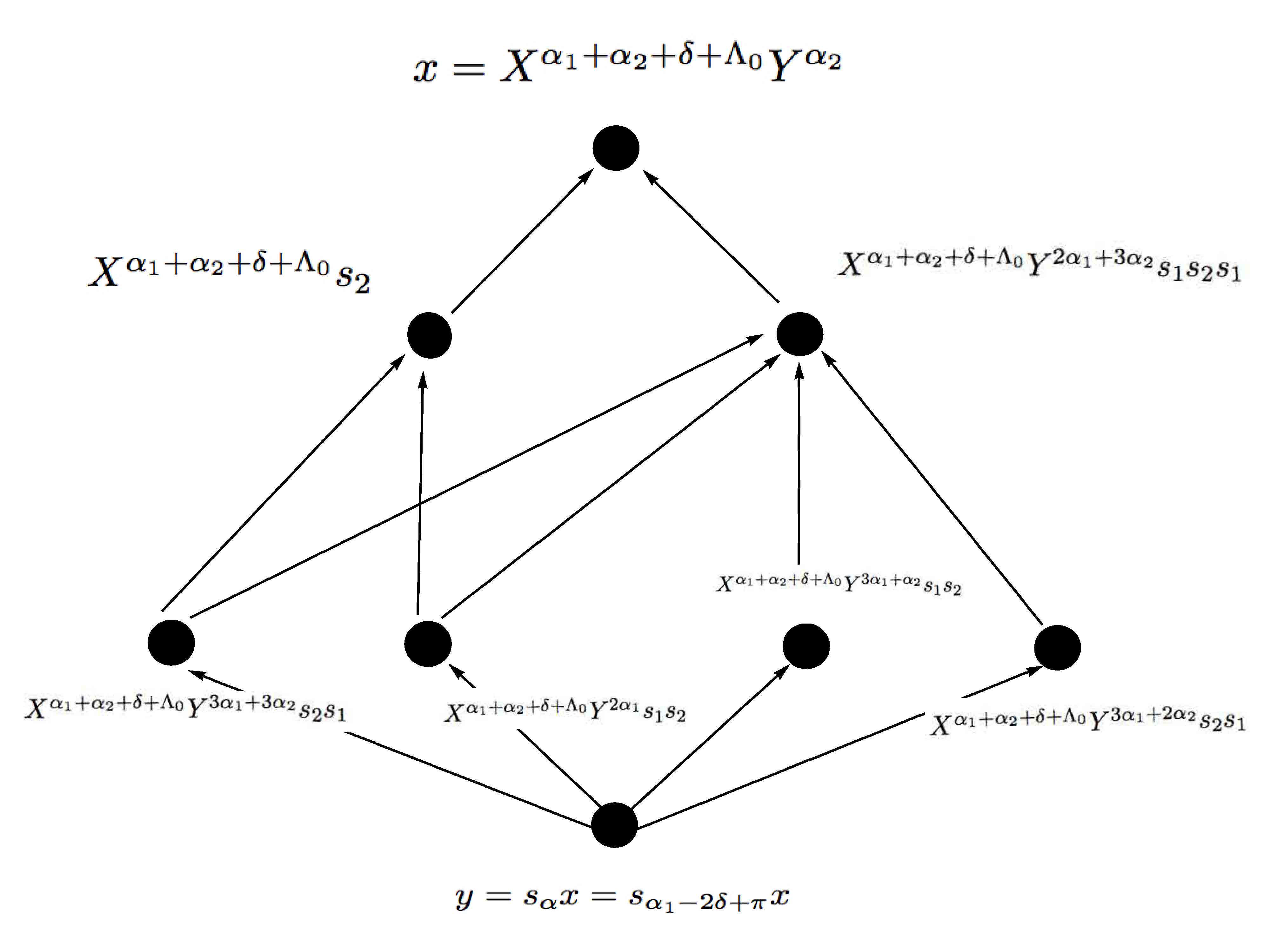}
\end{center}
\end{ex}

\begin{enumerate}
\item $x = X^{\alpha_1 + \alpha_2 + \delta +\Lambda_0}Y^{\alpha_2}$

\item $s_{\theta - 3\delta + \pi}x = X^{\alpha_1 + \alpha_2 + \delta + \Lambda_0}Y^{2\alpha_1 + 3\alpha_2}s_1s_2s_1$, a cocover of $x$

\item $s_{-\alpha_2 + \delta}x = X^{\alpha_1 + \alpha_2 + \delta +\Lambda_0}s_2$, a cocover of $x$

\item $s_{\alpha_1+\alpha_2 - 4\delta + 2\pi}s_{\alpha}x = X^{\alpha_1 + \alpha_2 + \delta +\Lambda_0}Y^{3\alpha_1 + \alpha_2}s_1s_2$, a cover of $s_{\alpha}x$

\item $s_{\theta - 3\delta + \pi} s_{\alpha}x = X^{\alpha_1 + \alpha_2 + \delta +\Lambda_0}Y^{2\alpha_1}s_1s_2$, a cover of $s_{\alpha}x$

\item $s_{-\alpha_2+\delta}s_{\alpha}x = X^{\alpha_1 + \alpha_2 + \delta +\Lambda_0}Y^{3\alpha_1 + 3\alpha_2}s_2s_1$, a cover of $s_{\alpha}x$

\item $s_{-\alpha_2 + \pi}s_{\alpha}x = X^{\alpha_1 + \alpha_2 + \delta +\Lambda_0}Y^{3\alpha_1 + 2\alpha_2}s_2s_1,$ a cover of $s_{\alpha}x$

\item $s_{\alpha}x = s_{\alpha_1 - 2\delta + \pi}x.$
\end{enumerate}

\begin{cor}\label{corner types}

Let $x = X^{\zeta} \tw$ with $\zeta \in \mathcal{T}$ and $\tw \in W_{\aff}$. If $\alpha = \nu + r\delta + j\pi$ corresponds to a corner of the graph $\Gamma_{x, \fin(\alpha)},$ then one of the following must hold:

\begin{enumerate}
\item $j = 0$
\item $j  =1$
\item $j = -\langle \zeta, \tal \rangle$
\item $j = -\langle \zeta, \tal \rangle - 1$.
\end{enumerate}

\end{cor}

%%%%%%%%%%%%%%%%%%%%%%%%%%%%%%%%%%%%%%%%%%%%%%

\subsection{Classification}

Using our new approach, we can bypass the bounds we needed on $\ell(\tw)$ and $ \ell(s_{\tv\tal}\tw)$ in Theorem \ref{thm: 1}, and we can reduce the bound we needed on $\langle \zeta, \alpha_i \rangle$.

\begin{thm}\label{thm: 2nd}
Let $x = X^{\tv \zeta} \tw$ and $y = s_{\alpha}x$ where $\alpha = -\tv \tal + j \pi$ is a positive double affine root and $\langle \zeta, \alpha_i \rangle > 2$ for $i = 0, 1, \dotsc, n$. Then $y$ is a cocover of $x$ if and only if one of the following holds: 

\begin{enumerate}
\item $j = 0$ and $\ell(\tv) = \ell(\tv s_{\tal}) + 1$.
\item $j  =1$ and $\ell(\tv) = \ell(\tv s_{\tal}) + 1 - \langle \tal, 2\rho \rangle$.
\item $j = \langle \zeta, \tal \rangle$ and $\ell(\tw^{-1} \tv s_{\tal}) = \ell(\tw^{-1} \tv) + 1.$
\item $j = \langle \zeta, \tal \rangle - 1$ and $\ell(\tw^{-1} \tv s_{\tal}) = \ell(\tw^{-1} \tv) + 1 - \langle \tal, 2\rho \rangle$.
\end{enumerate}
\end{thm}

\begin{proof}

Following the proof to Theorem \ref{thm: 1} we write $y$ in two different forms. 
\begin{align*}
y &= s_{\alpha}x\\
&=X^{\tv s_{\tal} (\zeta - j \tal)} s_{\tv \tal}\tw \\
&=X^{\tv (\zeta - (\langle \zeta,  \tal \rangle - j)  \tal)} s_{\tv \tal} \tw
\end{align*}

Using Proposition \ref{prop: 6} and the fact that $\zeta$ is dominant and regular, we have 
$$\ell(x) = \langle \zeta, 2\rho \rangle - \ell(\tw^{-1}\tv) + \ell(\tv).$$

If $y$ is a cocover of $x$, then $\alpha$ is a corner relative to $x$, and by Corollary \ref{corner types}, there are four possibilities for $j$:

\begin{enumerate}
\item $j = 0$ and $y = X^{\tv s_{\tal} \zeta} s_{\tv \tal} \tw$
\item $j  =1$ and $y = X^{\tv s_{\tal} (\zeta - \tal)} s_{\tv \tal} \tw$
\item $j = - \langle \tv \zeta, -\tv \tal \rangle = \langle \zeta, \tal \rangle$ and $y = X^{\tv \zeta} s_{\tv \tal} \tw$
\item $j = - \langle \tv \zeta, -\tv \tal \rangle  - 1 = \langle \zeta, \tal \rangle - 1$ and $y = X^{\tv (\zeta - \tal)} s_{\tv \tal} \tw.$
\end{enumerate}

So no matter which direction we are proving, we may reduce to these four cases. Using $\langle \tal, \tbe \rangle \leq 2$ for all $\tal, \tbe \in \Phi_{\aff}$ \cite[VI 1.3]{B} and the assumption that $\langle \zeta, \alpha_i \rangle > 2$ for $i = 0, 1, \dotsc, n$, we have that $\zeta - \tal$ is dominant and regular.

 Case (1): Let $j = 0$. Then $y = X^{\tv s_{\tal} \zeta} s_{\tv \tal} \tw$, and by using Proposition \ref{prop: 6} we have 
\begin{align*}
\ell(y) = & \langle \zeta, 2\rho \rangle - \ell( \tw^{-1} s_{\tv \tal} \tv s_{\tal}) + \ell(\tv s_{\tal})\\
= & \langle \zeta, 2\rho \rangle - \ell( \tw^{-1} \tv s_{\tal} \tv^{-1} \tv s_{\tal}) + \ell(\tv s_{\tal})\\
= & \langle \zeta, 2\rho \rangle - \ell( \tw^{-1} \tv) + \ell(\tv s_{\tal}).
\end{align*}

So $\ell(x) - \ell(y) = \ell(\tv) - \ell(\tv s_{\tal}),$ and $y$ is a cocover of $x$ if and only if $\ell(\tv) - \ell(\tv s_{\tal}) = 1$.\\

Case (2): Let $j = 1$. Then $y = X^{\tv s_{\tal} (\zeta - \tal)} s_{\tv \tal} \tw$, and by using Proposition \ref{prop: 6} we have  
\begin{align*}
\ell(y) = & \langle \zeta - \tal, 2\rho \rangle - \ell( \tw^{-1} s_{\tv \tal} \tv s_{\tal}) + \ell(\tv s_{\tal})\\
= & \langle \zeta, 2\rho \rangle - \langle \tal, 2\rho \rangle - \ell( \tw^{-1} \tv s_{\tal} \tv^{-1} \tv s_{\tal}) + \ell(\tv s_{\tal})\\
= & \langle \zeta, 2\rho \rangle - \langle \tal, 2\rho \rangle - \ell( \tw^{-1} \tv) + \ell(\tv s_{\tal}).
\end{align*}

So $\ell(x) - \ell(y) = \ell(\tv) - \ell(\tv s_{\tal}) +  \langle \tal, 2\rho \rangle,$ and $y$ is a cocover of $x$ if and only if $\ell(\tv) - \ell(\tv s_{\tal})  + \langle \tal, 2\rho \rangle= 1$.\\

 Case (3): Let $j = \langle \zeta, \tal \rangle$. Then $y = X^{\tv \zeta} s_{\tv \tal} \tw$, and by using Proposition \ref{prop: 6} we have 
\begin{align*}
\ell(y) & =  \langle \zeta, 2\rho \rangle - \ell(\tw^{-1} s_{\tv \tal} \tv) + \ell(\tv)\\
& = \langle \zeta, 2\rho \rangle - \ell(\tw^{-1} \tv s_{\tal}) + \ell(\tv).
\end{align*}

So $\ell(x) - \ell(y) = \ell(\tw^{-1} \tv s_{\tal}) - \ell(\tw^{-1} \tv),$ and $y$ is a cocover of $x$ if and only if $\ell(\tw^{-1} \tv s_{\tal}) - \ell(\tw^{-1} \tv) = 1$.\\

 Case (4): Let $j = \langle \zeta, \tal \rangle - 1$. Then $y = X^{\tv (\zeta - \tal)} s_{\tv \tal} \tw$, and by using Proposition \ref{prop: 6} we have 
\begin{align*}
\ell(y) & =  \langle \zeta, 2\rho \rangle - \langle \tal, 2\rho \rangle - \ell(\tw^{-1} s_{\tv \tal} \tv) + \ell(\tv)\\
& = \langle \zeta, 2\rho \rangle - \langle \tal, 2\rho \rangle - \ell(\tw^{-1} \tv s_{\tal}) + \ell(\tv).
\end{align*}

So $\ell(x) - \ell(y) = \ell(\tw^{-1} \tv s_{\tal}) - \ell(\tw^{-1} \tv) + \langle \tal, 2\rho \rangle,$ and $y$ is a cocover of $x$ if and only if $\ell(\tw^{-1} \tv s_{\tal}) - \ell(\tw^{-1} \tv) + \langle \tal, 2\rho \rangle= 1$.
\end{proof}

\subsection*{Further Work}

The next step in examining the Bruhat intervals is to determine if our classification will work when $\Phi_\fin$ is not simply laced. We expect that a similar classification exists, and that our method of using graphs and the length difference set can be applied. The first step in extending our work is to determine if it still holds that a length difference of one classifies a cocover relationship in the case where $\Phi_{\fin}$ is not simply laced. 

\subsection*{Acknowledgements}

This work was part of the author's doctoral dissertation at Virginia Polytechnic Institute and State University. The author would like to thank her advisor, Daniel Orr, for his help on this project. We also thank Dinakar Muthiah for suggesting the approach we used based on corners to classifying covers in the double affine Bruhat order.

\end{document}